\documentclass[12pt]{amsart}

\usepackage{amsthm}
\usepackage{amsmath}
\usepackage{mathrsfs}
\usepackage{amssymb}
\usepackage[dvipsnames]{xcolor}
\usepackage{tikz-cd}
\usepackage{enumitem}
\usepackage{hyperref}
\usepackage{algorithm}
\usepackage{algpseudocode}
\usepackage{graphicx}
\usepackage{adjustbox}
\hypersetup{colorlinks = true,
	linkcolor = MidnightBlue,
	urlcolor = BrickRed,
	citecolor = MidnightBlue}

\title[index obstructions on abelian varieties]{Refined index obstructions for Brauer classes on an abelian variety}
\author{Eoin Mackall}
\address{Department of Mathematics, University of California Santa Cruz, Santa Cruz, CA USA}
\email{eoinmackall@gmail.com}
\date{\today}

\newtheorem{thm}{Theorem}[section]

\newtheorem{prop}[thm]{Proposition}

\newtheorem{lem}[thm]{Lemma}

\theoremstyle{definition}
\newtheorem{defn}[thm]{Definition}
\newtheorem{exmp}[thm]{Example}
\newtheorem{rmk}[thm]{Remark}
\newtheorem*{conj*}{Conjecture}

\begin{document}
	
\begin{abstract}
We produce refined index obstructions, generalizing recently constructed index obstructions due to de Jong and Perry, for topologically trivial Brauer classes on smooth and projective complex varieties. We show that our refined obstructions are more stringent than previous obstructions and, as a consequence, we produce more counterexamples to the integral Hodge conjecture. Throughout this work, we focus on algorithmic aspects of these obstructions and we illustrate many of these aspects through the concrete examples of complex abelian varieties.
\end{abstract}
\maketitle

\section{Introduction}
\subsection{Background} The integration and application of Hodge theory to the area of Brauer groups and Azumaya algebras has had tremendous success in recent years.
The most substantial of this success is certainly the recent proof of the period-index conjecture for abelian threefolds due to Hotchkiss and Perry, \cite{hotchkiss2024periodindexconjectureabelianthreefolds}, but there have also been new constructions of counterexamples to the integral Hodge conjecture due to Hotchkiss \cite{hotchkiss2022hodgetheorytwistedderived} and, separately, de Jong--Perry \cite{dejong2022periodindexproblemhodgetheory}, that one can directly point to. Notable as well is the pioneering work of Kresch \cite{MR1934439} on representability of Brauer classes by quaternion Azumaya algebras that anticipated this development.

The point of this paper is to further elucidate the algorithmic nature of this connection that has gone relatively unnoticed or, at least, has lacked direct focus by the broader algebra community. Specifically, we focus on the \textit{explicit} and \textit{exact} methods of de Jong and Perry \cite{dejong2022periodindexproblemhodgetheory} that enabled their constructions of new counterexamples to the integral Hodge conjecture. While this aspect of the authors' work was not their main goal, we believe that there should be an immediate and direct appeal to researchers and users of Azumaya algebras.

To explain the parts of the work of de Jong--Perry on which we build, we focus solely on the case of a smooth and projective variety $X$ over the complex numbers. We suppose that one knows an explicit description for the structure of both the singular cohomology $\mathrm{H}^*(X^{an},\mathbb{Q})$ and its subgroup of integral classes, as well as the space of rational Hodge classes $\mathrm{Hdg}^*(X^{an})\subset \mathrm{H}^{2*}(X^{an},\mathbb{Q})$. The first point of this work is that:
\begin{enumerate}[label=(P\arabic*), ref=P\arabic*]
\item\label{intro: t1} the subgroup of $\mathrm{Br}(X)$, consisting of Brauer classes in the kernel of the forgetful map \[\mathrm{Br}(X)\rightarrow \mathrm{Br}_{top}(X^{an}),\quad [\mathcal{A}]\mapsto [\mathcal{A}_{top}],\] provides a collection of Brauer classes $\alpha$ on $X$ whose invariants are amenable to computation solely from the data of a rational $B$-field from $\mathrm{H}^2(X^{an},\mathbb{Q})$ and known structures on $\mathrm{H}^*(X^{an},\mathbb{Q})$.

\end{enumerate}

For those \textit{topologically trivial} Brauer classes $\alpha$ where (\ref{intro: t1}) applies, de Jong and Perry provide an algorithmic method for computing lower bounds for (and sometimes completely determining) the \textit{index of $\alpha$} from the data of a $B$-field representing $\alpha$. More precisely, de Jong and Perry show (cf.\ \cite[Theorem 5.10]{dejong2022periodindexproblemhodgetheory}) that the index $\mathrm{ind}(\alpha)$ divides an integer $d\geq 0$ only if there exist rational Hodge classes $c_j\in \mathrm{H}^{2j}(X^{an},\mathbb{Q})$ for all $0\leq j \leq m=\min\{d,\dim(X)\}$ such that the polynomials \begin{equation}\label{intro: polys}p_i^{B,d}(1, c_1,...,c_i) := \binom{d}{i} B^i + \sum_{j=1}^i\binom{d-j}{i-j}B^{i-j}c_j\end{equation} are integral classes inside $\mathrm{H}^{2i}(X^{an},\mathbb{Q})$ for all $1\leq i \leq m$. The second point of this work is that:

\begin{enumerate}[label=(P\arabic*), ref=P\arabic*, start=2]
\item\label{intro: t2} in concrete examples, e.g.\ for complex abelian varieties, one can algorithmically parametrize all such rational Hodge classes that satisfy this condition (cf.\ Remark \ref{rmk: classsolver}).
\end{enumerate}

Thus, one may effectively and algorithmically produce lower bounds for the indices of topologically trivial Brauer classes by way of (\ref{intro: t2}). This work largely started as a means of generalizing these computations with the goal of applying them to other aspects of Azumaya algebras.

\subsection{Refined Obstructions}
In \cite{dejong2022periodindexproblemhodgetheory}, the authors point out that any algebraic condition that one can impose on the generic fiber of the relative Severi--Brauer scheme for an Azumaya algebra --- representing the topologically trivial Brauer class $\alpha\in \mathrm{Br}(X)$ of a smooth projective complex variety $X$ --- can be used to deduce conditions analogous to but possibly more general than \cite[Theorem 5.10]{dejong2022periodindexproblemhodgetheory}. Since there is a corpus of research in this direction, \cite{MR1348794,MR1615533,MR3359722,MR3590349,MR3722500,MR3990787,MR4078228}, it's natural to investigate this further.

Motivated by results in the Chow groups of Severi--Brauer varieties, we make an orthogonal observation to the point of de Jong and Perry. Namely, there is still more that one can say about the constraints of \cite[Theorem 5.10]{dejong2022periodindexproblemhodgetheory}, even from a strictly topological perspective. Specifically we observe the following ``refined" index obstructions:

\begin{thm}\label{intro: thm}
Let $X$ be a smooth and projective complex variety. Let $\alpha\in \mathrm{Br}(X)$ be a topologically trivial Brauer class. Let $\mathcal{A}$ be an Azumaya algebra in the class $\alpha$. Write $\pi:\mathbf{SB}(\mathcal{A})\rightarrow X$ for the associated Severi--Brauer scheme. Suppose that $\mathrm{ind}(\alpha)=n$. 

Then there exists a coherent $\mathcal{O}_{\mathbf{SB}(\mathcal{A})}$-module $\mathcal{M}$ and there exists a class $z\in K_0^{top}(X^{an})$ such that \[c_k(\mathcal{M})=\sum_{j=0}^k \binom{n-j}{k-j}\pi^*c_j(z)h^{k-j}\in \mathrm{H}^{2k}(\mathbf{SB}(\mathcal{A})^{an},\mathbb{Q})\] where $h=c_1(\mathcal{O}_{\mathbb{P}(E)}(1))$ under the topological identification of $\mathbf{SB}(\mathcal{A})$ with $\mathbb{P}(E)$ for a complex topological vector bundle $E$ on $X$.

In particular, if the index $\mathrm{ind}(\alpha)$ divides an integer $d>0$, then there must exist a choice of rational Hodge classes $c_j\in \mathrm{H}^{2j}(X^{an},\mathbb{Q})$ for all $0\leq j \leq m=\min\{d,\dim(X)\}$ such that for all $l\geq 1$ we have \begin{equation}\label{intro: Steenrod} \mathrm{Sq}^{2l}(C_i)\equiv \sum_{k=0}^l\binom{i-l+k-1}{k}C_{l-k}C_{i+k} \pmod{2}\end{equation} for the Steenrod operations $\mathrm{Sq}^{2l}:\mathrm{H}^{2i}(X^{an},\mathbb{Z}/2\mathbb{Z})\rightarrow \mathrm{H}^{2(i+l)}(X^{an},\mathbb{Z}/2\mathbb{Z})$ of the resulting integral cohomology classes $C_i=p_i^{B,d}(1,c_1,...,c_i)$.
\end{thm}

Thus, if one can algorithmically compute the Steenrod operations on the left of (\ref{intro: Steenrod}) then, combined with a parameterization from (\ref{intro: t2}), one gets a collection of further polynomial constraints on the allowable residual classes of the polynomials (\ref{intro: polys}) of these Hodge classes modulo $2$. There is also nothing particularly special about the prime $2$ here: for each prime $p>2$, the reduced power operations $\mathrm{P}^l(C_i)$ of the classes $C_i$ must satisfy certain universal weighted homogeneous polynomials in the classes $C_k$ as well, so if one can compute $\mathrm{P}^l(C_i)$, then these can also be used to provide further constraints.

\begin{rmk}
In \cite{hotchkiss2022hodgetheorytwistedderived}, Hotchkiss provides an alternative obstruction to bounding the index of a topologically trivial Brauer class utilizing twisted Mukai structures. When one can easily work with the image of the Chern character map, as is the case for abelian varieties, Hotchkiss's obstruction is even more restrictive than Theorem \ref{intro: thm}. See Remark \ref{rmk: Hotchkiss} for more information on this alternative obstruction.
\end{rmk}

Theorem \ref{intro: thm} is the culmination of a sequence of results from Section \ref{sec: obs} of this article. Throughout this section, we focus on the particularly nice case of a very general complex abelian variety for motivation. The section culminates in an example (Example \ref{ex: obsdjp}) where the refined index obstructions from Theorem \ref{intro: thm} can be explicitly shown to provide stricter constraints than the index obstructions of de Jong and Perry \cite[Theorem 5.10]{dejong2022periodindexproblemhodgetheory} alone. We mention that, along the same vein of the results in ibid., this example also provides another counterexample to the integral Hodge conjecture.

\subsection{Applications}
Compared to Section \ref{sec: obs} of this article, where our focus is on building the refined index obstructions of Theorem \ref{intro: thm}, in the final section of this article, Section \ref{sec: apps}, we focus on the application of the algorithmic aspects of this work in more detail. 

Among the results in this section, we include a simple analysis (Proposition \ref{prop: bounds}) on the degrees of failure of the de Jong--Perry index obstruction algorithm for topologically trivial Brauer classes $\alpha$ on a complex abelian variety $X$ whose period $\mathrm{per}(\alpha)$ is not coprime to $(\dim(X)-1)!$. We also provide some data (see Figures \ref{fig: bfield_data_4_4}--\ref{fig: bfield_data_8_2_r}) to the effect of how sharp this analysis is, and to how it may also apply to our new refined index obstruction algorithm.

We also include an example (Example \ref{ex: indec}), found with the aid of a computer, of an indecomposable topologically trivial Brauer class $\alpha\in \mathrm{Br}(X)$ of period $\mathrm{per}(\alpha)=2$ and index $\mathrm{ind}(\alpha)=4$ on a very general abelian threefold $X$. While indecomposable Brauer classes with these invariants are most likely not new, we remark that it is interesting that they can be constructed so readily with the help of the de Jong--Perry and our refined index obstruction algorithms.

\subsection*{Notation}
\begin{itemize}
\item Given a complex algebraic variety $X$, we write $X^{an}$ for the analytification of $X$.
\item For a group $G$ and a topological space $Y$, we write $\mathrm{H}^i(Y,G)$ for the singular cohomology of $Y$ with coefficients in $G$.
\item For a paracompact, Hausdorff, locally contractible space (such as $X^{an}$ for a smooth complex variety $X$), we identify the derived cohomology $\mathrm{H}^i(X,\underline{G})$ of a constant sheaf associated to a group $G$ with the singular cohomology $\mathrm{H}^i(X,G)$.
\end{itemize}

\subsection*{Acknowledgments} The author would like to thank Alex Perry for comments on an initial draft of this paper.

\section{Preliminaries}\label{sec: pre}
Let $X$ be a complex abelian variety of dimension $g$. It follows from uniformization that $X^{an}\cong \mathbb{C}^g/\Lambda$ for a lattice $\Lambda\subset \mathbb{C}^g$ of rank $2g$. Naturally this gives isomorphisms \[\mathrm{H}_1(X^{an},\mathbb{Z})\cong \Lambda\quad\mbox{and}\quad\mathrm{H}^1(X^{an},\mathbb{Z})\cong \mathrm{Hom}_{Ab}(\Lambda, \mathbb{Z}).\] Choosing a basis for $\Lambda$, we can identify the topological space of $X$ with a product of $2g$ circles. The K{\"u}nneth formula then gives a decomposition \[\bigwedge^k \mathrm{H}^1(X^{an},\mathbb{Z}) \cong \mathrm{H}^k(X^{an},\mathbb{Z})\] sending a $k$-form to the associated cup product on the right hand side. In this section and below, we will frequently fix a principal polarization $\theta\in \mathrm{H}^2(X^{an},\mathbb{Z})$ and a symplectic basis $x_1,y_1,...,x_g,y_g$ for $\mathrm{H}^1(X^{an},\mathbb{Z})$ so that $\theta=x_1\wedge y_1+\cdots + x_g\wedge y_g$ in this basis. For more details on this set-up, we refer the reader to \cite{MR2062673}.

For an algebraic variety $Y$, we write $\mathrm{Br}(Y)$ for the Brauer group that consists of equivalence classes of Azumaya algebras on $Y$. Analogously, for any topological space $Z$, we write $\mathrm{Br}_{top}(Z)$ for the similarly defined (topological) Brauer group that consists of equivalence classes of topological Azumaya algebras on $Z$. By \cite[Corollaire 1.7]{MR244269}, there is a canonical isomorphism $\mathrm{Br}_{top}(Z)\cong \mathrm{H}^3(Z,\mathbb{Z})_{tors}$ between the topological Brauer group of any finite CW-complex $Z$ (e.g.\ for $Z=X^{an}$) and the torsion subgroup of the third singular cohomology group of $Z$.

Fix an integer $n\geq 1$. From the following commutative diagram of sheaves on the topological space $X^{an}$,
\[\begin{tikzcd}
0\arrow{r} & \underline{\mathbb{Z}} \arrow["\cdot n"]{r}\arrow{d} & \underline{\mathbb{Z}} \arrow["\mathrm{exp}\left(\frac{2\pi i}{n} \cdot -\right)"]{r}\arrow["\cdot \frac{2\pi i}{n}"]{d} &\underline{\mu_n}\arrow{r}\arrow{d} & 1\\
0\arrow{r} & \underline{\mathbb{Z}}\arrow["\cdot 2\pi i"]{r} & \mathcal{O}_{X^{an}}\arrow["\mathrm{exp}"]{r}& \mathcal{O}_{X^{an}}^*\arrow{r} & 1
\end{tikzcd}
\] we get the commutative diagram below, with exact rows as well.
\begin{equation}\label{eq: exp}\begin{tikzcd}
\mathrm{H}^2(X^{an},\mathbb{Z})\arrow["\mathrm{exp}\left(\frac{2\pi i}{n} \cdot -\right)"]{r}\arrow["\cdot \frac{2\pi i}{n}"]{d} & \mathrm{H}^2(X^{an},\mu_n)\arrow{r}\arrow{d} & \mathrm{H}^3(X^{an},\mathbb{Z})\arrow[equals]{d}\\
\mathrm{H}^2(X^{an},\mathcal{O}_{X^{an}})\arrow{r} & \mathrm{H}^2(X^{an},\mathcal{O}_{X^{an}}^*)\arrow{r} & \mathrm{H}^3(X^{an},\mathbb{Z})
\end{tikzcd}\end{equation} Since $\mathrm{H}^2(X^{an},\mu_n)$ is a torsion group, its image in $\mathrm{H}^3(X^{an},\mathbb{Z})$ lands in the topological Brauer group $\mathrm{Br}_{top}(X^{an})$ under the above mentioned isomorphism. The following lemma is immediate.

\begin{lem}\label{lem: toptriv}
Let $X$ be a smooth projective variety over $\mathbb{C}$. Then the following diagram of canonical morphisms is commutative.
\[\begin{tikzcd}
\mathrm{H}^2_{\acute{e}t}(X,\mu_n)\arrow[twoheadrightarrow]{r}\arrow["\sim"]{d} & \mathrm{Br}(X)[n]\arrow["\phi"]{rd}\arrow{d} & \\
\mathrm{H}^2(X^{an},\mu_n)\arrow{r} & \mathrm{H}^2(X^{an},\mathcal{O}_{X^{an}}^*)_{tors}\arrow{r} & \mathrm{H}^3(X^{an},\mathbb{Z})_{tors}
\end{tikzcd}\]
Moreover, under the identification $\mathrm{Br}_{top}(X^{an})\cong \mathrm{H}^3(X^{an},\mathbb{Z})_{tors}$ the map $\phi$ is identified with the forgetful map \[\mathrm{Br}(X)[n]\rightarrow \mathrm{Br}_{top}(X^{an}),\quad [\mathcal{A}]\mapsto [\mathcal{A}_{top}]\] which sends the class of an Azumaya algebra $\mathcal{A}$ on $X$ to the topological Azumaya algebra $\mathcal{A}_{top}:=\mathcal{A}\otimes_{\mathcal{O}_X} \mathscr{C}(X^{an},\mathbb{C})$.$\hfill\square$
\end{lem}

For an abelian variety $X$ of dimension $g$, we observed above that $\mathrm{H}^k(X^{an},\mathbb{Z})$ is a free group of rank $\binom{2g}{k}$ and, so, it is also torsion free. Thus the topological Brauer group $\mathrm{Br}_{top}(X^{an})$ is trivial, and the top row of (\ref{eq: exp}) gives a concrete description of the group $\mathrm{Br}(X)[n]$.

\begin{rmk}
For a general smooth and projective variety $Y$ over $\mathbb{C}$, the Brauer classes $\alpha\in \mathrm{Br}(Y)$ that are in the kernel of the forgetful map $\mathrm{Br}(Y)\rightarrow \mathrm{Br}_{top}(Y^{an})$ are said to be topologically trivial. For any Azumaya algebra $\mathcal{A}$ with $\alpha=[\mathcal{A}]$ a topologically trivial class, it follows from Lemma \ref{lem: toptriv} that $\mathcal{A}_{top}$ is a neutral Azumaya algebra, i.e.\ isomorphic to $\mathcal{E}nd(V)$ for a complex topological vector bundle $V$ on $Y^{an}$.
\end{rmk}

Kresch \cite{MR1934439} observed that if a Brauer class, associated to a $2$-form $\beta\in \mathrm{H}^2(Y^{an}, \mathbb{Z}/2\mathbb{Z})$ through Lemma \ref{lem: toptriv} for a smooth projective complex variety $Y$, was represented by a quaternion Azumaya algebra, then this necessarily implies the existence of algebraic cycles on $Y$ which may be obstructed by the integral Hodge structure of $Y$. 

Almost two decades later, Hotchkiss \cite{hotchkiss2022hodgetheorytwistedderived} realized that knowledge of the explicit Hodge structures on algebraic invariants associated to $Y$ (e.g.\ on cohomology $\mathrm{H}^*(Y^{an},\mathbb{Q})$ or topological $K$-theory $K^{top}_0(Y^{an})_{\mathbb{Q}}$) could be used to prove more general period-index bounds for Brauer classes on a smooth projective complex variety than those described by Kresch in the period 2 case. Hotchkiss then used these ideas to produce new counterexamples to the integral Hodge conjecture by explicitly producing integral Hodge classes on relative Severi--Brauer schemes that were not algebraic.

Hotchkiss's method for producing these kinds of counterexamples to the integral Hodge conjecture was later formalized in work of de Jong and Perry \cite{dejong2022periodindexproblemhodgetheory}. Although their main motivation for formalizing this process was toward results on bounding the index of a Brauer class as a function of the dimension of the underlying variety and the period, de Jong and Perry actually produce an incredibly explicit method for bounding the index of a given topologically trivial Brauer class utilizing an explicitly describable Hodge structure.

\begin{thm}\label{thm: hodge}
Let $X$ be any smooth and projective variety over $\mathbb{C}$. Assume that $\beta\in \mathrm{H}^2(X^{an},\mathbb{Z}/n\mathbb{Z})$ defines a topologically trivial Brauer class $\alpha\in \mathrm{Br}(X)[n]$ and write $B=b/n\in \mathrm{H}^2(X^{an},\mathbb{Q})$ for the rational B-field gotten from any lift $b\in\mathrm{H}^2(X^{an},\mathbb{Z})$ of $\beta$. Then:
\begin{enumerate}
\item For any Azumaya algebra $\mathcal{A}$ of class $\alpha$, the associated Severi--Brauer scheme $\pi:\mathbf{SB}(\mathcal{A})\rightarrow X$ is topologically isomorphic to the projectivization $\mathbb{P}(E)$ of a complex topological vector bundle $E$ on $X$.
\item Under the identification of $\pi$ with the projective bundle map $\mathbb{P}(E)\rightarrow X$, there is a complex topological line bundle $\mathcal{O}_{\mathbb{P}(E)}(1)$. Writing $h=c_1(\mathcal{O}_{\mathbb{P}(E)}(1))\in \mathrm{H}^2(X^{an},\mathbb{Z})$ for the first Chern class of $\mathcal{O}_{\mathbb{P}(E)}(1)$, there is an induced isomorphism for any $d\in \mathbb{Z}$ \[\sum_j \pi^*(-) \cup h^j : \bigoplus_{j=0}^r \mathrm{H}^{d-2j}(X^{an},\mathbb{Z})\xrightarrow{\sim} \mathrm{H}^d(\mathbf{SB}(\mathcal{A})^{an},\mathbb{Z})\] where $r=\mathrm{deg}(\mathcal{A})-1$ is the relative dimension of $\pi$.
\item The class $nh+\pi^*b \in \mathrm{H}^2(\mathbf{SB}(\mathcal{A})^{an},\mathbb{Z})$ is algebraic, and there is an isomorphism of rational Hodge structures for any $d\in \mathbb{Z}$ \[\sum_j \pi^*(-) \cup (h+\pi^*B)^j : \bigoplus_{j=0}^r \mathrm{H}^{d-2j}(X^{an},\mathbb{Q})(-j)\xrightarrow{\sim} \mathrm{H}^d(\mathbf{SB}(\mathcal{A})^{an},\mathbb{Q})\] where $r$ is as above.
\end{enumerate}
\end{thm}

\begin{proof}
This is \cite[Lemma 5.9]{dejong2022periodindexproblemhodgetheory}.
\end{proof}

It follows from Theorem \ref{thm: hodge} that, given any integral Hodge class $\gamma\in \mathrm{H}^{2k}(\mathbf{SB}(\mathcal{A})^{an}, \mathbb{Q})$, there exist rational Hodge classes $c_j\in \mathrm{H}^{2j}(X^{an},\mathbb{Q})$ such that the classes \begin{equation}\label{eq: polyhodge}p_i^{B,k}(c_0, c_1,...,c_i) := \binom{k}{i} B^ic_0 + \sum_{j=1}^i\binom{k-j}{i-j}B^{i-j}c_j\in \mathrm{H}^{2i}(X^{an},\mathbb{Q})\end{equation} are integral, i.e.\ in the image of $\mathrm{H}^{2i}(X^{an},\mathbb{Z})$ under the canonical map to $\mathrm{H}^{2i}(X^{an},\mathbb{Q})$. Indeed, if $\gamma$ is a Hodge class then it has an expansion \[\gamma = \pi^*(c_0)(h+\pi^*B)^k+\pi^*(c_1)(h+\pi^*B)^{k-1}+\cdots + \pi^*(c_m)(h+\pi^*B)^{k-m}\] where $m=\min\{k,\dim(X)\}$. Writing $\gamma$ in the basis given by powers of $h$ yields the formulae (\ref{eq: polyhodge}), which are necessarily integral if $\gamma$ is integral. This exact condition was used to give lower bounds on the index of $\alpha$ in \cite[Theorem 5.10]{dejong2022periodindexproblemhodgetheory}:

\begin{thm}[{\cite[Theorem 5.10]{dejong2022periodindexproblemhodgetheory}}]\label{thm: dpindexbounds} Let $X$ be a smooth, projective, complex variety. Let $\alpha\in \mathrm{Br}(X)$ be a topologically trivial Brauer class.

Then if $\mathrm{ind}(\alpha)$ divides $d>0$, there must exist rational Hodge classes $c_j\in \mathrm{H}^{2j}(X^{an},\mathbb{Q})$ for integers $1\leq j\leq m=\min\{d,\dim(X)\}$ such that the polynomials $p_i^{B,d}$ of (\ref{eq: polyhodge}) are integral for all $i=1,...,m$ with $c_0=1$.
\end{thm}

\begin{proof}
The point is that, if $\mathrm{ind}(\alpha)$ divides $d$, then there exists a relative Severi--Brauer variety $Y$ defined over an open subset $U\subset X$ such that both $Y$ embeds in $\mathbf{SB}(\mathcal{A})|_U$, in relative codimension $d$ over $U$, and such that the induced embeddings on fibers over $U$ are linear embeddings. The closure of $Y$ in $\mathbf{SB}(\mathcal{A})$ then defines a cycle whose cohomology class is an integral Hodge class $\gamma\in \mathrm{H}^d(\mathbf{SB}(\mathcal{A})^{an},\mathbb{Q})$ with $c_0=1$.
\end{proof}

\begin{rmk}\label{rmk: classsolver}
If the Hodge structure on $X$ is known explicitly, then Theorem \ref{thm: dpindexbounds} can often be checked algorithmically. For example, if $X$ is a very general abelian variety, then the Hodge structure on $\mathrm{H}^{2k}(X^{an},\mathbb{Q})$ is determined by the principal polarization $\theta$. The rational Hodge space $\mathrm{Hdg}^k(X^{an}, \mathbb{Q})$ is the span of $\theta^k$ and the integral Hodge classes are the integer multiples of $\theta^k/k!$, cf.\ \cite[Theorem 17.4.1]{MR2062673}.

Suppose that $c_1,...,c_i$ are rational Hodge classes with $c_j\in \mathrm{H}^{2j}(X^{an},\mathbb{Q})$ for all $j=1,...,i$ and with some $c_0\in \mathbb{Z}$ fixed. We can write $c_j=a_j\theta^j$ for some rational numbers $a_j$. Expanding (\ref{eq: polyhodge}) in a basis for $\mathrm{H}^{2i}(X^{an},\mathbb{Z})$, we are asking if there exists a $\binom{2g}{i}\times i$ rational system $P_i$ solved by an $i$-vector $\bar{a}$ of rational numbers $a_j$ such that \begin{equation}\label{eq: algsmith}\bar{z} = P_i\cdot \bar{a}+\bar{c}_0\end{equation} where the vector $\bar{z}$ has integer coefficients and $\bar{c}_0$ is set from $\binom{d}{i}B^ic_0$. Let $K$ be any matrix whose rows form a $\mathbb{Q}$-basis for the orthogonal complement to the column space of $P_i$.

Multiplying by $K$ gives the linear system \[K\bar{z} = K\bar{c}_0\] which has an integral solution $\bar{z}$ if and only if $\bar{z}-\bar{c}_0$ is in the column space of $P_i$, i.e.\ if such an $\bar{a}$ solution exists. Clearing the denominators of $K$, we can assume that $K$ is integral. 

We can then algorithmically solve for any integer solutions to this equation: replacing $K$ with its Smith normal form $K=LDR$ with $L,R$ unimodular integer matrices, we get \[D(R\bar{z}) = L^{-1}K\bar{c}_0\] and since $D$ is diagonal, this amounts to solving equations $D_{j,j}\cdot r_j = s_j$ where $r_j=(R\bar{z})_j$ is free and $s_j=(L^{-1}K\bar{c}_0)_j$ is specified. Solutions $r_j$ then exist if and only if $D_{j,j}$ divides $s_j$.

Note also that this allows one to parametrize the integer lattice of solutions to (\ref{eq: algsmith}) assuming that at least one such solution exists. Indeed, the conditions $D_{j,j}$ uniquely specify the $r_j$ for all $1\leq j \leq \ell-t$ for some $t\geq 0$. The remaining $t\geq 0$ terms in $R\bar{z}$ are allowed to be free.
\end{rmk}

\begin{rmk}\label{rmk: karpenko}
In \cite[\S5.6]{dejong2022periodindexproblemhodgetheory}, the authors point out that any algebraic cycle that should exist on the generic fiber of the Severi--Brauer scheme $\pi:\mathbf{SB}(\mathcal{A})\rightarrow X$ can be used as a condition like in Theorem \ref{thm: dpindexbounds} above, but where $c_0$ is changed appropriately. For example, it's known that if the Brauer class $\alpha$ has index $n$, then the generic fiber of $\pi$ will contain a cycle of degree $n/\gcd(n,k)$ in codimension $k$, see \cite{MR1327303}, and these conditions may be more general than Theorem \ref{thm: dpindexbounds} alone.
\end{rmk}

\section{Reduced power operations and index bounds}\label{sec: obs}
In this section we give a refinement of the theorem of de Jong and Perry for bounding the index of a topologically trivial Brauer class on a smooth projective complex variety $X$. Afterwards, we consider the possible obstructions to the period-index conjecture granted by both the theorem of de Jong and Perry, and this new refinement, in the case of a very general abelian variety. We show that, in this particular case, these obstructions actually all vanish. We conclude this section with some more examples of the failure of the integral Hodge conjecture.

The starting point for our refinement is the observation that, for at least some canonical choice of a hypothetical class satisfying the de Jong--Perry obstruction of Theorem \ref{thm: dpindexbounds}, the rational Hodge classes that one finds $c_j\in \mathrm{H}^{2j}(X^{an},\mathbb{Z})$ are not allowed to be arbitrary.

\begin{lem}\label{lem: taut}
Let $X$ be a smooth and projective complex variety. Let $\alpha\in \mathrm{Br}(X)$ be a topologically trivial Brauer class. Let $\mathcal{A}$ be an Azumaya algebra in the class $\alpha$. Write $\pi:\mathbf{SB}(\mathcal{A})\rightarrow X$ for the associated Severi--Brauer scheme.

Then there is a vector bundle $\zeta_{\mathcal{A}}$ on $\mathbf{SB}(\mathcal{A})$ of rank $d=\mathrm{deg}(\mathcal{A})$ whose Chern classes satisfy the formula \[c_k(\zeta_{\mathcal{A}}) = \sum_{j=0}^k \binom{d-j}{k-j}\pi^*c_j(E)h^{k-j}\] for all $k\geq 0$. Here we use the identification of Theorem \ref{thm: hodge} that $\mathbf{SB}(\mathcal{A})$ is topologically $\mathbb{P}(E)$ and $h=c_1(\mathcal{O}_{\mathbb{P}(E)}(1))$.
\end{lem}

\begin{proof}
The bundle $\zeta_{\mathcal{A}}$ in the lemma statement is the middle term in the Euler sequence for the Severi--Brauer variety $\mathbf{SB}(\mathcal{A})$ over $X$, \[0\rightarrow \mathcal{O}_{\mathbf{SB}(\mathcal{A})}\rightarrow \zeta_{\mathcal{A}}\rightarrow \mathcal{T}_{\mathbf{SB}(\mathcal{A})/X}\rightarrow 0.\] The dual bundle $\zeta_{\mathcal{A}}^\vee$ is defined in such a way that the fiber over a complex point of $s\in \mathbf{SB}(\mathcal{A})$ corresponds tautologically to a left-ideal of rank $d$ inside the ring $\mathcal{A}_{\pi(s)} \otimes_{\mathcal{O}_{X,\pi(s)}} \mathbb{C}$. Choosing an isomorphism $\mathcal{A}_{top}\cong\mathcal{E}nd(E)$ for a complex topological vector bundle $E$, one identifies these left-ideals with the collection of matrices whose columns span the line $L\subset \pi^*(E)_{s}$ determined by $s\in \pi^{-1}(\pi(s))$. One identifies the bundle $\pi^*(E)^\vee\otimes \mathcal{O}_{\mathbb{P}(E)}(-1)$ similarly as having fibers over $s$ the vector space $\mathrm{Hom}(\pi^*(E)_s,L)$, which is exactly the above ideal.

Hence, for all $k\geq 0$, there is a relation $c_k(\zeta_{\mathcal{A}})=c_k(\pi^*(E)\otimes \mathcal{O}_{\mathbb{P}(E)}(1))$ inside of the group $\mathrm{H}^{2k}(\mathbf{SB}(\mathcal{A})^{an},\mathbb{Z})$. Expanding the Chern class of the tensor product using known formulas proves the claim.
\end{proof}

\begin{prop}\label{prop: taut}
Let $X$ be a smooth and projective complex variety. Let $\alpha\in \mathrm{Br}(X)$ be a topologically trivial Brauer class. Let $\mathcal{A}$ be an Azumaya algebra in the class $\alpha$. Write $\pi:\mathbf{SB}(\mathcal{A})\rightarrow X$ for the associated Severi--Brauer scheme.

Then there exists a coherent $\mathcal{O}_{\mathbf{SB}(\mathcal{A})}$-module $\mathcal{M}$ such that:
\begin{enumerate}
\item there exists an open subset $U\subset X$ such that $\mathcal{M}_U$ is locally free of rank $\mathrm{ind}(\alpha)$,
\item\label{it: cind} for each $k\geq 0$, the Chern class $c_k(\mathcal{M})$ is an algebraic class that restricts on the fibers of $\pi$ as a class of degree $\binom{\mathrm{ind}(\alpha)}{k}$.
\end{enumerate}
In particular, $c_{\mathrm{ind}(\alpha)}(\mathcal{M})$ is an algebraic class that restricts on the fibers of $\pi:\mathrm{SB}(\mathcal{A})\rightarrow X$ as the class of a linear space.
\end{prop}

\begin{proof}
Set $d=\mathrm{deg}(\mathcal{A})$, $n=\mathrm{ind}(\alpha)$, and let $\eta\in X$ be the generic point. The central simple algebra $\mathcal{A}_{\mathbb{C}(\eta)}\cong M_{d/n}(D)$ is a matrix algebra over a $\mathbb{C}(\eta)$-division algebra $D$. Let $M$ be the simple right $\mathcal{A}_{\mathbb{C}(\eta)}$-module $D^{d/n}$, which has dimension $nd$.

Picking a basis $\{e_i\}_{i=1}^{nd}$ for $M$, we can construct an $\mathcal{O}_X$-submodule $\mathcal{N}\subset j_*M$, where $j:\eta\rightarrow X$ is the inclusion, as the image sheaf of the canonical map $\mathcal{O}_X^{nd}\rightarrow j_*M$. Then $\mathcal{N}$ is coherent since $X$ is Noetherian, and by construction $\mathcal{N}_{\mathbb{C}(\eta)}\cong M$.

Consider the multiplication map $\mathcal{N}\otimes_{\mathcal{O}_X} \mathcal{A}\rightarrow j_*M$. The image of this map is a coherent sheaf $\mathcal{M}'$, again since $X$ is Noetherian, which is naturally a right $\mathcal{A}$-module. Let $\zeta^\vee_{\mathcal{A}}$ be the tautological bundle of left $\pi^*\mathcal{A}$-ideals of rank $d$ on $\mathbf{SB}(\mathcal{A})$. Set $\mathcal{M}''= \pi^*\mathcal{M}'\otimes_{\pi^*{\mathcal{A}}} \zeta^\vee_\mathcal{A}$.

One can check that, over the generic fiber $\pi^{-1}(\eta)=\mathbf{SB}(\mathcal{A}_{\mathbb{C}(\eta)})$, the sheaf $\mathcal{M}''_{\pi^{-1}(\eta)}$ is a sub-bundle of $\zeta^\vee_{\mathcal{A}_{\mathbb{C}(\eta)}}$ of rank $n$. From this observation and Lemma \ref{lem: taut}, we get the claim by taking $\mathcal{M}:=(\mathcal{M}'')^\vee$.
\end{proof}

We want to describe the Chern classes $c_k(\mathcal{M})$ of the coherent sheaf $\mathcal{M}$ constructed in Proposition \ref{prop: taut} for $0\leq k \leq \mathrm{ind}(\alpha)$ in more detail. To do this, we note that $[\mathcal{M}]\in K_0(\mathbf{SB}(\mathcal{A}))$ is, by construction, the dual of the image of the class $[\mathcal{M}']$ under the Quillen transformation on $K$-groups $K_0(X,\mathcal{A})\rightarrow K_0(\mathbf{SB}(\mathcal{A}))$. This map is compatible with the natural maps to topological $K$-theory, so we find that there is a class $x\in K^{top}_0(X^{an},\mathcal{A}_{top})$ which maps to $[\mathcal{M}]_{top}\in K_0^{top}(\mathbf{SB}(\mathcal{A})^{an})$.

Since $\alpha$ is topologically trivial, the class $x$ comes from $K_0^{top}(X^{an})$ via the natural transformation sending a vector bundle $F$ to $F\otimes E$, where $\mathcal{A}_{top}\cong \mathcal{E}nd(E)$ for a complex topological bundle $E$. Altogether, there is a class $y=[F_0]-[F_1]\in K_0^{top}(X^{an})$ for two complex topological vector bundles $F_0,F_1$ on $X^{an}$ so that \[c_t(\mathcal{M})=c_t(\pi^*(F_0\otimes E)^\vee\otimes_{\mathcal{E}nd(E)} \zeta_{\mathcal{E}nd(E)})\cdot c_t(\pi^*(F_1\otimes E)^\vee\otimes_{\mathcal{E}nd(E)} \zeta_{\mathcal{E}nd(E)})^{-1},\] where $c_t$ denotes the total Chern polynomial. Using the identification $\zeta_{\mathcal{E}nd(E)}^\vee \cong \pi^*E^\vee\otimes \mathcal{O}_{\mathbb{P}(E)}(-1)$, this simplifies to \begin{align*}c_t(\mathcal{M})&=c_t(\pi^*F_0^\vee \otimes \mathcal{O}_{\mathbb{P}(E)}(1))c_t(\pi^*F_1^\vee \otimes \mathcal{O}_{\mathbb{P}(E)}(1))^{-1}\\ &= c_t(\pi^*(y^\vee)\cdot [\mathcal{O}_{\mathbb{P}(E)}(1)])\end{align*} In particular, for each $k\geq 0$ we have \[ c_k(\mathcal{M})=\sum_{j=0}^k \binom{n-j}{k-j}\pi^*c_j(y^\vee)h^{k-j}\] for a class $y\in K^{top}_0(X^{an})$ and where $h=c_1(\mathcal{O}_{\mathbb{P}(E)}(1))$. We have proved:

\begin{thm}\label{thm: chernm}
Let $X$ be a smooth and projective complex variety. Let $\alpha\in \mathrm{Br}(X)$ be a topologically trivial Brauer class. Let $\mathcal{A}$ be an Azumaya algebra in the class $\alpha$. Write $\pi:\mathbf{SB}(\mathcal{A})\rightarrow X$ for the associated Severi--Brauer scheme. 

Let $\mathcal{M}$ be the coherent $\mathcal{O}_{\mathbf{SB}(\mathcal{A})}$-module constructed in the proof of Proposition \ref{prop: taut}. Then there exists a class $z\in K_0^{top}(X^{an})$ such that \begin{equation}\label{eq: chern}c_k(\mathcal{M})=\sum_{j=0}^k \binom{n-j}{k-j}\pi^*c_j(z)h^{k-j}\end{equation} where $h=c_1(\mathcal{O}_{\mathbb{P}(E)}(1))$ under the topological identification of $\mathbf{SB}(\mathcal{A})$ with $\mathbb{P}(E)$ described in Theorem \ref{thm: hodge}. $\hfill\square$
\end{thm}

\begin{rmk}
Note that, in (\ref{eq: chern}), the formula \[\binom{n-j}{k-j}=\frac{(n-j)(n-j-1)\cdots (n-k+1)}{(k-j)!}\] is a generalized binomial coefficient. So, if $k-j<0$, then this is $0$. If $k=n$, then it is $1$. And if $n-j\leq 0$, this formula still makes sense.
\end{rmk}

These observations are what will allow our refinement:

\begin{prop}\label{prop: obs2}
Let $X$ be a smooth and projective complex variety. Let $z\in K_0^{top}(X^{an})$ be any class. Then for all $j\geq 0$ we have \begin{equation}\label{eq: steenrod}\mathrm{Sq}^{2j}(c_i(z))= \sum_{k=0}^j\binom{i-j+k-1}{k}c_{j-k}(z)c_{i+k}(z)\end{equation} where $\mathrm{Sq}^{2j}:\mathrm{H}^{2i}(X^{an},\mathbb{Z}/2\mathbb{Z})\rightarrow \mathrm{H}^{2i+2j}(X^{an},\mathbb{Z}/2\mathbb{Z})$ is the $j$th Steenrod operation modulo $2$. Similarly, for any odd prime $p$, we have \begin{equation}\label{eq: powerop}\mathrm{P}^j(c_i(z))=Q_{p,i,j}(c_1(z),...,c_{i+j(p-1)}(z))\end{equation} for a uniquely determined weighted homogeneous polynomial $Q_{p,i,j}$ of degree $i+j(p-1)$, where the variables $c_i(z)$ are weighted of degree $i$, and where $\mathrm{P}^j:\mathrm{H}^{2i}(X^{an},\mathbb{Z}/p\mathbb{Z})\rightarrow \mathrm{H}^{2i+2j(p-1)}(X^{an},\mathbb{Z}/p\mathbb{Z})$ is the $j$th reduced power operation modulo $p$.
\end{prop}

\begin{proof}
These formulae are classical for Chern classes of a vector bundle, see \cite[Theorem 11.3]{MR58213}. In general, we can always write $z=[F]-N$ for some vector bundle $F$ and some integer $N\geq 0$. The claim is then immediate from the equalities $c_t(z)=c_t(F)c_t(N)^{-1}=c_t(F)$.
\end{proof}

In particular, if $\alpha$ is a topologically trivial Brauer class on a smooth projective complex variety $X$ such that the cohomology ring $\mathrm{H}^*(X^{an},\mathbb{Z})$ is torsion-free, then if $\mathrm{ind}(\alpha)$ divides $d$, there must exist rational Hodge classes $c_j\in \mathrm{H}^{2j}(X^{an},\mathbb{Q})$ such that the polynomials $p_i^{B,d}(c_1,...,c_i)$ are integral classes for all $i\leq d$, as in Theorem \ref{thm: dpindexbounds}. 

But, moreover, for \textit{at least one collection} of rational Hodge classes $c_j\in \mathrm{H}^{2j}(X^{an},\mathbb{Q})$ making all of the polynomials $p_i^{B,d}(c_1,...,c_i)$ integral, the resulting classes $p_i^{B,d}(c_1,...,c_i)$ are required to be the Chern classes of some element $z\in K_0^{top}(X^{an})$ by Theorem \ref{thm: chernm}. Hence, for such classes $\{c_j\}_{j=1}^\infty$, the reduced integral classes $p_i^{B,d}(c_1,...,c_i)\in \mathrm{H}^{2i}(X^{an},\mathbb{Z}/p\mathbb{Z})$ must satisfy the Steenrod (if $p=2$), or the reduced power (if $p>2$), operation relations of Proposition \ref{prop: obs2}.

\begin{rmk}\label{rmk: Hotchkiss}
In \cite{hotchkiss2022hodgetheorytwistedderived}, Hotchkiss provides a different obstruction from that of de Jong--Perry using $B$-twisted Mukai structures. Let $X$ be a smooth projective complex variety and let $\alpha\in \mathrm{Br}(X)$ be a topologically trivial Brauer class determined by a $B$-field $B\in \mathrm{H}^2(X^{an},\mathbb{Q})$. Then Hotchkiss observes that in order for $\mathrm{ind}(\alpha)$ to divide $d$, there must exist some $v\in K_0^{top}(X^{an})$ and rational Hodge classes $H_i\in \mathrm{H}^{2i}(X^{an},\mathbb{Q})$ for $i=1,...,\dim(X)$ such that \[\mathrm{exp}(B)\mathrm{ch}(v)=d+H_1+H_2+\cdots +H_{\dim(X)} \] where $\mathrm{ch}(v)$ is the Chern character of $v$.

There are two na\"ive obstructions to the existence of such a class $v\in K_0^{top}(X^{an})$. Since the Chern character is an isomorphism rationally, there exists $w\in \mathrm{K}_0^{top}(X^{an})_{\mathbb{Q}}$ such that $\mathrm{ch}(w)=d+H_1+\cdots + H_{\dim(X)}$. If $l\in K_0^{top}(X^{an})_\mathbb{Q}$ is such that $\mathrm{ch}(l)=\mathrm{exp}(-B)$, then $\mathrm{ch}(w\cdot l) = \mathrm{ch}(v)$.

From this one can formally determine that \begin{equation}\label{eq: hotchobs}c_k(v) = \binom{d}{k}(-B)^{k} + \sum_{i=1}^k c_i(w)\binom{d-i}{k-i}(-B)^{k-i}.\end{equation} Using Newton's identities relating power sum polynomials and elementary symmetric polynomials, we find that $c_i(w)=q_i(H_1,...,H_i)$ for some weighted homogeneous polynomial $q_i$ with rational coefficients. Thus, a first obstruction to the existence of $v\in K_0^{top}(X^{an})$ is that there must exist Hodge classes $Q_i=q_i(H_1,...,H_i)$ for all $i=1,...,\dim(X)$ such that $p_k^{-B,d}(Q_1,...,Q_k)$ are integral for all $k=1,...,\dim(X)$.

Granted that one can find rational Hodge classes $Q_i$ for all $i$ as above, a second obstruction is gotten from the fact that if there exists a $v\in K_0^{top}(X^{an})$ with $c_k(v)=p_k^{-B,d}(Q_1,...,Q_k)$, then the classes $\mathrm{ch}_k(v)$, which are formal polynomials in the classes $c_k(v)$, must be in the image of the Chern character map. Note that this is a genuine restriction as, for example, in the case of an abelian variety the Chern character of a class $v\in K_0^{top}(X^{an})$ must be integral.

If both of these obstructions vanish, then there is a $v\in K_0^{top}(X^{an})$ whose Chern classes $c_k(v)$ satisfy (\ref{eq: hotchobs}) for all $k$. In particular, if both of Hotchkiss's obstructions vanish, then so do the obstructions of de Jong--Perry (from Theorem \ref{thm: dpindexbounds}) and so do our refined obstructions.
\end{rmk}

\subsection{Abelian varieties} Our goal now is to provide an explicit example where one can check that the cohomological obstructions described above are stronger than the obstructions of de Jong--Perry, described in Theorem \ref{thm: dpindexbounds}. This is done in Example \ref{ex: obsdjp} below. We point out that any such example (where the de Jong--Perry obstruction vanishes, but where our obstruction does not) provides another counterexample to the integral Hodge conjecture by \cite[Theorem 5.10]{dejong2022periodindexproblemhodgetheory}.

\begin{exmp}\label{ex: 0alg}
Let $X$ be an abelian variety. Then for all $j>0$ and for all $x\in \mathrm{H}^i(X^{an},\mathbb{Z}/2\mathbb{Z})$ we have $\mathrm{Sq}^{j}(x)=0$, regardless of $i>0$. Indeed:
\begin{enumerate}
\item $\mathrm{Sq}^j(x)=0$ if $j>1$ and $x\in \mathrm{H}^1(X^{an},\mathbb{Z}/2\mathbb{Z})$,
\item $\mathrm{Sq}^1(x)=x^2$ if $i=1$, and $x^2=0$ since $\mathrm{H}^*(X^{an},\mathbb{Z})$ is isomorphic to a space of alternating forms, see Section \ref{sec: pre},
\item using linearity and the Cartan formula, one gets $\mathrm{Sq}^j(x)=0$ for all $x\in \mathrm{H}^i(X^{an},\mathbb{Z}/2\mathbb{Z})$ $i>1$ as well.
\end{enumerate}
Similarly, for any odd prime $p$, the reduced power operation $P^j(x)=0$ for all $x\in \mathrm{H}^i(X^{an},\mathbb{Z}/p\mathbb{Z})$ for all $j>0$ and $i>0$.
\end{exmp}

\begin{exmp}\label{ex: obsdjp}
Let $X$ be a very general abelian fourfold. Keep the notation of Section \ref{sec: pre}, so $\mathrm{H}^1(X^{an},\mathbb{Z})$ has a symplectic basis $\{x_i,y_i\}_{i=1}^4$. Let $b\in \mathrm{H}^2(X^{an},\mathbb{Z})$ be the form \[ b=x_1\wedge(y_1+y_3)+x_2\wedge y_2 + x_3\wedge y_1.\] The $B$-field $b/2$ defines a Brauer class $\alpha\in \mathrm{Br}(X)$ of period $2$. We will show that the de Jong--Perry obstructions vanish in degree $4$, so that there is no obstruction to $\mathrm{ind}(\alpha)$ dividing $4$ coming from Theorem \ref{thm: dpindexbounds}. However, the obstructions of Proposition \ref{prop: obs2} are apparent in degree $4$. Since the symbol length $\ell(\alpha)=3$ (in the notation of \cite[Definition 23.1 (2)]{hotchkiss2024periodindexconjectureabelianthreefolds}), this implies $\mathrm{ind}(\alpha)=8$.

The de Jong--Perry obstructions require the existence of rational Hodge classes $c_1,c_2,c_3,c_4$ that produce integral classes: \begin{align*}p_1^{B,4}(c_1)&=4B+c_1\\
p_2^{B,4}(c_1,c_2)&=6B^2+3Bc_1+c_2\\
p_3^{B,4}(c_1,c_2,c_3)&=4B^3+3B^2c_1+2Bc_2+c_3\\
p_4^{B,4}(c_1,c_2,c_3,c_4)&=B^4+B^3c_1+B^2c_2+Bc_3+c_4.\end{align*}
Since $4B$ is integral, the first of these says that $c_1$ is integral, i.e.\ there is an integer $m_1$ such that $c_1=m_1\theta$ for the principal polarization $\theta$. Similarly, $6B^2$ is integral, but $3Bc_1$ is not. One can check that \[b\theta = 2\omega_1\wedge \omega_2 + (\omega_1+\omega_2)\wedge(\omega_3+\omega_4) + (x_1\wedge y_3+x_3\wedge y_1)\wedge (\omega_2+\omega_4)\] where $\omega_i=x_i\wedge y_i$, which is linearly independent from $\theta^2$. So $3m_1/2\in \mathbb{Z}$ and $c_2=m_2(\theta^2/2)$ for an integer $m_2\in \mathbb{Z}$.

In the third condition, both $4B^3$ and $2Bc_2$ are integral. Additionally, since $m_1$ must be even, $3B^2c_1$ is integral. Hence we find $c_3=m_3(\theta^3/6)$ for some integer $m_3\in \mathbb{Z}$. In the fourth condition $B^4=0$ since $b$ is a sum of three symbols, and one can check that $B^2c_2=0$. The remaining summands are: \begin{align*}B^3c_1&=-\frac{3}{4}m_1 (\omega_1\wedge\omega_2\wedge\omega_3\wedge\omega_4)\\ Bc_3 & = m_3(\omega_1\wedge \omega_2\wedge \omega_3\wedge \omega_4).\end{align*} If we write $c_4=m_4(\theta^4/24)$ we find $-(3/4)m_1+m_3+m_4\in \mathbb{Z}$. Since $m_3$ is an integer, and $m_1$ is an even integer, we find $m_4\in (1/2)\mathbb{Z}$. Altogether, the possible solutions are given by the parameter set \begin{equation}\label{eq: soln} m_1\in 2\mathbb{Z},\quad m_2\in \mathbb{Z},\quad m_3\in \mathbb{Z}, m_4\in (1/2)\mathbb{Z}.\end{equation}

Let $c_1,c_2,c_3,c_4$ be any choice of rational Hodge classes as above, satisfying (\ref{eq: soln}). For each of $i=1,2,3,4$ set $C_i\equiv p_i^{B,4} \pmod{2}$ using these classes as input. By Proposition \ref{prop: obs2} and Example \ref{ex: 0alg}, we have \[0\equiv \mathrm{Sq}^2(c_2)=C_1C_2+C_3 \pmod{2}. \]
We now calculate the reduction modulo $2$ for the classes $p_i^{B,4}$ compared to these constraints. Since $m_1\in 2\mathbb{Z}$, we have \[C_1=p_1^{B,4}(c_1)\equiv 2b+c_1 \equiv 0 \pmod{2}.\] We also have \[C_3=p_3^{B,4}(c_1,c_2,c_3)\equiv (b^3/2)+(m_1/2)(b^2\theta/2)+m_2(b\theta^2/2)+m_3(\theta^3/6) \pmod{2}.\] Using the notation above, we can write this as \begin{multline*}C_3\equiv \omega_1\wedge \omega_2\wedge \omega_3 + (m_1/2)(\omega_1\wedge (\omega_3+\omega_2)\wedge \omega_4+(x_1\wedge y_3 + x_3\wedge y_1)\wedge\omega_2\wedge\omega_4)\\ +m_2(\omega_1\wedge\omega_3\wedge\omega_4+\omega_2\wedge\omega_3\wedge\omega_4 + (x_1\wedge y_3+x_3\wedge y_1)\wedge \omega_2\wedge\omega_4) \\+m_3(\omega_1\wedge \omega_2\wedge \omega_3 + \omega_1\wedge \omega_2\wedge \omega_4 + \omega_1\wedge \omega_3\wedge \omega_4 + \omega_2\wedge \omega_3\wedge \omega_4) \pmod{2}.\end{multline*} In order to cancel the leading $\omega_1\wedge\omega_2 \wedge \omega_3$ term in $C_3$, we must have $m_3\equiv 1 \pmod{2}$. Similarly, looking at the terms \[\omega_1\wedge \omega_2\wedge \omega_4,\quad \omega_1\wedge \omega_3\wedge \omega_4,\quad \omega_2\wedge \omega_3\wedge \omega_4\] we find congruences \begin{align*}(m_1/2)+m_3 &\equiv 0 \pmod{2}\\ (m_1/2)+m_2+m_3 &\equiv 0 \pmod{2}\\ m_2+m_3&\equiv 0 \pmod{2}.\end{align*} The last of these imply $m_2\equiv 1 \pmod{2}$, and then the middle and first imply that $m_1/2$ is both odd and even.

Hence for all choices of $m_1,m_2,m_3$, we have $C_3 \not\equiv 0 \pmod{2}$, and there do not exist rational Hodge classes $c_1,c_2,c_3,c_4$ making the classes $p_1^{B,4}, p_2^{B,4}, p_3^{B,4}, p_4^{B,4}$ both integral and, simultaneously, subject to the Steenrod constraints of Proposition \ref{prop: obs2}.
\end{exmp}

\begin{rmk}
For an abelian variety, the Chern character induces an isomorphism between $K_0^{top}(X^{an})$ and $\mathrm{H}^{2*}(X^{an},\mathbb{Z})$. Thus, in order for the two obstructions of Remark \ref{rmk: Hotchkiss} to vanish, one needs:
\begin{enumerate}
\item to find rational Hodge classes $Q_i$ for all $i=1,...,\dim(X)$ such that $p_k^{-B,d}(Q_1,...,Q_k)$ is integral for all $k=1,...,\dim(X)$;
\item to check that when one uses Newton's recursive formulas for writing the Chern classes $c_k(v)=p_k^{-B,d}(Q_1,...,Q_k)$ as Chern characters $\mathrm{ch}_k(v)$, the resulting classes are integral. 
\end{enumerate}
We note for more general varieties, however, determining the image of the Chern character map may be a difficult task itself. In such cases, one can still use our refined obstructions without this information.
\end{rmk}

\section{Additional applications to Azumaya algebras}\label{sec: apps}
In this section we highlight some applications to the area of Azumaya algebras that are directly accessible through work of Hotchkiss \cite{hotchkiss2022hodgetheorytwistedderived}, the obstructions of de Jong--Perry \cite{dejong2022periodindexproblemhodgetheory}, and the results of this paper. We focus on the interpretation of all of these methods as an algorithm for producing lower bounds on the index for a given Brauer class on an abelian variety, and how this can be applied to produce new examples.

\subsection{Algorithmic index bounds}
Let $X$ be a smooth and projective complex variety. Let $\alpha\in \mathrm{Br}(X)$ be a topologically trivial Brauer class on $X$. Then Theorem \ref{thm: dpindexbounds} can be interpreted as an algorithm for producing lower bounds on the index of $\alpha$:
\begin{enumerate}
\item one iterates over multiples of the period $\mathrm{per}(\alpha)$,
\item one checks whether such rational Hodge classes exist, e.g.\ using
Remark \ref{rmk: classsolver},
\item a lower bound on $\mathrm{ind}(\alpha)$ is obtained by the largest multiple where the previous step fails.
\end{enumerate}

Further, when one produces a parameterization of the integral classes $p_i^{B,d}$, e.g.\ as in Remark \ref{rmk: classsolver}, one gets a parameterization for all mod $p$ inputs, for a prime $p$, to the obstruction of Proposition \ref{prop: obs2}. In these cases, this new obstruction is therefore also algorithmically checkable, since this space will always be finite. 

It's interesting to know when these algorithms fail. For an abelian variety, these algorithms will always vanish in degree $2\mathrm{per}(\alpha)^{\ell(\alpha)}$ where $\ell(\alpha)$ is the symbol length of $\alpha$, cf.\ \cite[Definition 23.1 (2)]{hotchkiss2024periodindexconjectureabelianthreefolds}, since the index $\mathrm{ind}(\alpha)$ must divide $\mathrm{per}(\alpha)^{\ell(\alpha)}$. A different bound is given by de Jong and Perry in \cite[Remark 5.14]{dejong2022periodindexproblemhodgetheory}. There the authors show that their obstructions vanish for any smooth and projective complex variety $X$, and for any topologically trivial Brauer class $\alpha\in \mathrm{Br}(X)$, in degree $2\mathrm{per}(\alpha)^{\dim(X)-1}$ if every prime divisor of the period $\mathrm{per}(\alpha)$ is coprime to $(\dim(X)-1)!$. 

For an abelian variety, however, we can show that these algorithms will fail much more often:

\begin{prop}\label{prop: bounds}
Let $X$ be a complex abelian variety. Let $\alpha\in \mathrm{Br}(X)$ be a topologically trivial Brauer class on $X$. Assume that $\mathrm{per}(\alpha)=p^r$ for a prime $p$ and some $r>0$. 

Then the algorithm of de Jong and Perry, for producing a lower bound on the index of $\alpha$, will fail in degree $2p^{rs}$ where \[rs = r\dim(X) - \left\lfloor\frac{\dim(X)-1}{p-1}\right\rfloor+\lfloor \log_p(\dim(X)-1)\rfloor + 1.\]
Note that when $p^r=2$, this is simply $rs=\lfloor\log_p(\dim(X)-1)\rfloor+2$.
\end{prop}

\begin{proof}
Let $b\in \mathrm{H}^2(X^{an},\mathbb{Z})$ be a $2$-form whose Brauer class is $\alpha$, and let $B=b/p^r$ be the associated rational $B$-field. We will show that the class \[\gamma = (h+\pi^*B)^{p^{rs}}\in \mathrm{H}^{2p^{rs}}(\mathbf{SB}(\mathcal{A})^{an}, \mathbb{Q}),\] in the notation of Theorem \ref{thm: hodge}, is an integral Hodge class, for $rs$ above. Thus, one can take $c_j=0$ for all $j>0$ in Theorem \ref{thm: dpindexbounds} in this degree.

Now we want to show that for each $i\geq 1$, the term $\binom{p^{rs}}{i}B^i$ is integral. By properties of the exterior algebra, we can write $b^i=i!a$ for an integral class $a$. The $p$-valuation for the coefficient of the class $a$ in this term is then \[v_p\left(\binom{p^{rs}}{i}i!\frac{1}{p^{ri}}\right)=rs-v_p(i)+v_p(i!)-ri=r(s-i)+v_p((i-1)!),\] using \cite[Lemma 3.5]{MR1615533} for the binomial coefficient. 

We want to show that, for $rs$ as in the statement, the valuation above is nonnegative for all $0\leq i\leq \dim(X)$. We recall from \cite[II Lemma 5.6]{MR1697859} that \[v_p((i-1)!)=\frac{i-1-s_p(i-1)}{p-1}\] where $s_p(i-1)$ is the sum of the base-$p$ digits of $i-1$. Isolating for $rs$ and substituting this formula, we are tasked with showing \[rs \geq ri - \frac{i-1-s_p(i-1)}{(p-1)}.\]

Note that for $i$ in the range $0\leq i \leq \dim(X)$ we have \[s_p(i-1)\leq (p-1)\lfloor \log_p(\dim(X)-1)+1\rfloor.\] Replacing this in the above, it suffices to take $rs$ such that \[
rs \geq r\dim(X) - \frac{\dim(X)-1}{(p-1)} + \lfloor\log_p(\dim(X)-1)+1\rfloor,\] which gives the claim.
\end{proof}

\begin{table}[htpb]
\centering
\renewcommand{\arraystretch}{1.2}
\begin{tabular}{|c||c|c|c|c|c|c|c|c|c|c|}
\hline
$\dim(X)$ & $2$ & $3$ & $2^2$ & $5$ & $7$ & $2^3$ & $3^2$ & $11$ & $13$ & $2^4$ \\
\hline
3 & - & - & -     & - & -  & - & - & - & -           & - \\
4 & 3 & - & $7/2$ & - & -  & $11/3$ & - & - & -      & $15/4$ \\
5 & 4 & - & $9/2$ & - & -  & $14/3$ & - & - & -      & $19/4$ \\
6 & 4 & - & 5      & - & - & $16/3$ & -  & - & -     & $11/2$ \\
7 & 4 & 6 & $11/2$ & - & -  & 6     & $13/2$ & - & - & $25/4$ \\
8 & 4 & 7 & 6      & - & -  & $20/3$ & $15/2$ & - & - & 7 \\
9 & 5 & 7 & 7       & - & - & $23/3$ & 8 & - & -      & 8 \\
10 & 5 & 9 & $15/2$ & - & - & $25/3$ & $19/2$ & - & - & $35/4$ \\
11 & 5 & 9 & 8      & - & - & 9     & 10 & - & -      & $19/2$ \\
12 & 5 & 10 & $17/2$ & - & - & $29/3$ & $11$ & - & -  & $41/4$ \\
\hline
\end{tabular}\vspace{0.4 cm}
\caption{Values of the rational bound $s$ from Proposition \ref{prop: bounds} for an abelian variety $X$ of dimension $\dim(X)$, listed in rows, and for a prime power $p^r$, listed in columns. Cells where $s\geq\dim(X)$ are denoted by a hyphen.}
\label{tab:s_values_rational}
\end{table}

\begin{rmk}
We note that one can obtain bounds where the de Jong--Perry obstruction fails, but for a Brauer class on an abelian variety of period $n$ with multiple prime factors, using both Proposition \ref{prop: bounds} above and the primary decomposition of $n$.
\end{rmk}

\begin{rmk}
For any $d>0$, and for any prime $p$ coprime to $(d-1)!$, there exists a Brauer class $\alpha\in \mathrm{Br}(X)$ on a product of $d$ pairwise non-isogenous elliptic curves $X$ with index $\mathrm{ind}(\alpha)=p^t$ where $1\leq t \leq d-1$, and $\mathrm{per}(\alpha)=p$, cf.\ \cite[Proposition 5.19]{dejong2022periodindexproblemhodgetheory}. 

For such Brauer classes $\alpha$, we must have that $p\geq d$ since $p\nmid(d-1)!$. It follows that the integral Hodge classes constructed in the proof of Proposition \ref{prop: bounds} provide no new examples of failure of the integral Hodge conjecture for these classes since our bound simplifies to \[s=d-\lfloor(d-1)/(p-1)\rfloor +\lfloor \log_p(d-1)\rfloor + 1 = d+1\] in this case.

The authors also construct classes $\alpha\in \mathrm{Br}(X)$ in \cite[Proposition 5.19]{dejong2022periodindexproblemhodgetheory} of a prime power period $\mathrm{per}(\alpha)=p^r$, for any $r>1$, and of index $\mathrm{ind}(\alpha)=p^{m}$, for any $1\leq t \leq d-1$, for some $m$ in $r(t-1)<m\leq rt$. Note that, despite the claim of the Proposition, the exact value of $m$ is undetermined in the proof of \cite[Proposition 5.19]{dejong2022periodindexproblemhodgetheory}. In these cases, new counterexamples to the integral Hodge conjecture would appear if $m>rs$ for the bound $rs$ of Proposition \ref{prop: bounds}. However, since $p\geq d$, this again fails in these cases since $rs\geq rd$ trivially.
\end{rmk}

Proposition \ref{prop: bounds} is particularly interesting in the case that $\mathrm{per}(\alpha)=2$. We tested the index obstructions of de Jong--Perry against the refined Chern index obstructions coming from the relations of Proposition \ref{prop: obs2} for Brauer classes $\alpha\in \mathrm{Br}(X)$ of period $\mathrm{per}(\alpha)\leq 5$ on very general abelian varieties $X$ with $\dim(X)\in [4,8]$. The results are displayed in Figures \ref{fig: bfield_data_4_4}--\ref{fig: bfield_data_8_2_r} below. We note that, from this data alone, it appears that the upper bounds from Proposition \ref{prop: bounds}, for the degrees where the de Jong--Perry obstructions fail, are not sharp; it also appears that the Chern obstructions from Proposition \ref{prop: obs2} vanish in these same degrees, although it isn't clear to the author why this should be true.

Regarding data collection, we used the algorithm of Remark \ref{rmk: classsolver}, and the outline at the beginning of this subsection, to systematically test the index of Brauer classes $\alpha\in \mathrm{Br}(X)$ for an abelian variety $X$ in each case. In each case we aimed to test around $100$ such classes $\alpha$, and in dimension $8$ we tested only $56$ classes. For $\alpha$ selection, we selected an equal number of $\alpha$ classes determined uniformly at random from equivalence classes of sets of $2$-forms $\{b,b+k\theta\}_k\in\mathrm{H}^2(X^{an},\mathbb{Z})$ where $1\leq k< \mathrm{per}(\alpha)$ and where $\theta$ was the principal polarization, among all forms $b$ with Hamming weight bounded by $\dim(X)+2$ (except in the case of $\mathrm{per}(\alpha)=2$ and $\mathrm{dim}(X)=6$ where we allowed Hamming weight up to $12$). Note that the machine running these computations (having access to 64G of RAM and 50G of swap) eventually ran out of memory on forms of higher Hamming weight.

\subsection{Indecomposable Azumaya algebras}
Algorithms that allow one to determine lower bounds on the index of a Brauer class $\alpha\in \mathrm{Br}(X)$ can also be used to analyze other interesting algebraic properties of the Brauer class $\alpha$, e.g.\ decomposability of $\alpha$. In this subsection, we show how we can leverage these algorithms of de Jong--Perry and the previous subsection to obstruct decomposability of a Brauer class.

\begin{defn}
Let $X$ be a smooth, projective variety. Let $\alpha\in \mathrm{Br}(X)$ be a Brauer class on $X$. We say that $\alpha$ is \textit{decomposable} if there exists a Brauer class $\beta\in \mathrm{Br}(X)$ such that both \[\mathrm{ind}(\beta)<\mathrm{ind}(\alpha),\quad \mathrm{ind}(\alpha-\beta)<\mathrm{ind}(\alpha).\] If no such $\beta$ exists, then we say that $\alpha$ is \textit{indecomposable}.
\end{defn}

\begin{rmk}
Let $X$ be a smooth projective variety. Then $\alpha\in \mathrm{Br}(X)$ is decomposable if and only if there exists an Azumaya algebra $\mathcal{A}$ in the class $\alpha$ and Azumaya algebras $\mathcal{B},\mathcal{C}$ on $X$ satisfying: \begin{enumerate}
\item$\mathrm{ind}(\mathcal{B})<\mathrm{ind}(\mathcal{A})$ and $\mathrm{ind}(\mathcal{C})<\mathrm{ind}(\mathcal{A})$,
\item and $\mathcal{A}\cong \mathcal{B}\otimes \mathcal{C}$.
\end{enumerate}
An Azumaya algebra $\mathcal{A}$ that admits $\mathcal{B},\mathcal{C}$ satisfying (1) and (2) above is said to be decomposable. Analogously, an Azumaya algebra that is not decomposable is said to be indecomposable. 

It follows that if $\alpha\in \mathrm{Br}(X)$ is indecomposable, then every Azumaya algebra $\mathcal{A}$ in the class of $\alpha$ is indecomposable.
\end{rmk}

For a long time, determining whether there existed indecomposable central simple algebras of prime period $p>2$ and index $p^2$ was an open problem. This problem was eventually resolved by Karpenko \cite{MR1348794}, who produced obstructions to decomposability coming from degrees of certain naturally defined algebraic cycles.

In the case of algebras of period $2$ and index $4$, a classical theorem of Albert \cite[XI \S6]{MR595} shows that all such algebras are products of two quaternion algebras (thus motivating the question for higher primes $p$). This theorem does not generalize to an a arbitrary Azumaya algebra over a more general scheme, however, as can be deduced from a theorem of Antieau-Auel-First \cite[Theorem A]{MR3908769} on the existence of such algebras lacking an involution of the first kind.

We give an example, in the same vein as \cite[Theorem A]{MR3908769}, of an indecomposable Brauer class $\alpha\in \mathrm{Br}(X)$ of period $2$ and index $4$ on an abelian threefold below (Example \ref{ex: indec}). It follows that no Azumaya algebra of degree $4$ representing this class admits such an involution, although we do not know if $\alpha$ is even represented by such an algebra.

Now let $X$ be a complex abelian variety. Then for any integral $2$-form $b\in \mathrm{H}^2(X^{an},\mathbb{Z})$ representing a topologically trivial class $\alpha\in \mathrm{Br}(X)$ of period $\mathrm{per}(\alpha)=n$ and index $\mathrm{ind}(\alpha)=m>1$, one gets an algorithm for testing if $\alpha$ is indecomposable as follows:
\begin{enumerate}
\item one iterates over elements $c$ of the finite group $\mathrm{H}^2(X^{an},\mathbb{Z}/n\mathbb{Z})$, with associated Brauer class $\gamma\in \mathrm{Br}(X)$, giving a decomposition $\bar{b}=c+(\bar{b}-c)$ where $\bar{b}$ is $b \pmod{n}$;
\item one checks that either the de Jong--Perry index obstructions or the Chern obstructions force either of the equalities $\mathrm{ind}(\gamma)=\mathrm{ind}(\alpha)$ or $\mathrm{ind}(\alpha-\gamma)=\mathrm{ind}(\alpha)$;
\item if the latter holds for all classes $c\in \mathrm{H}^2(X^{an},\mathbb{Z}/n\mathbb{Z})$, then it follows that $\alpha$ is indecomposable.
\end{enumerate}

\begin{exmp}\label{ex: indec}
Let $X$ be a very general abelian threefold. Keeping the notation of Section \ref{sec: pre}, let \[b=x_1\wedge(y_1+y_2)+x_2\wedge(y_1+y_3)+x_3\wedge(y_1+y_2+y_3)\] and consider the topologically trivial Brauer class $\alpha\in \mathrm{Br}(X)$ of period $2$ determined from the rational $B$-field $B=b/2$. Then the index of $\alpha$ is $\mathrm{ind}(\alpha)=4$ (one can bound the index from below by $4$ using index obstructions, and then it must be $4$ by \cite[Theorem 1.3]{hotchkiss2024periodindexconjectureabelianthreefolds}). 

We used the above procedure, iterating over all $32,768$ elements of $\mathrm{H}^2(X^{an},\mathbb{Z}/2\mathbb{Z})$, to find $\alpha$ is indecomposable. This form was also used as an example in \cite[Example 23.4]{hotchkiss2024periodindexconjectureabelianthreefolds} and it has symbol length $\ell(\alpha)=3$.
\end{exmp}

\bibliographystyle{amsalpha}
\bibliography{bib}

\begin{figure}[h]
    \begin{adjustbox}{center}
    \includegraphics[width=1.6\textwidth]{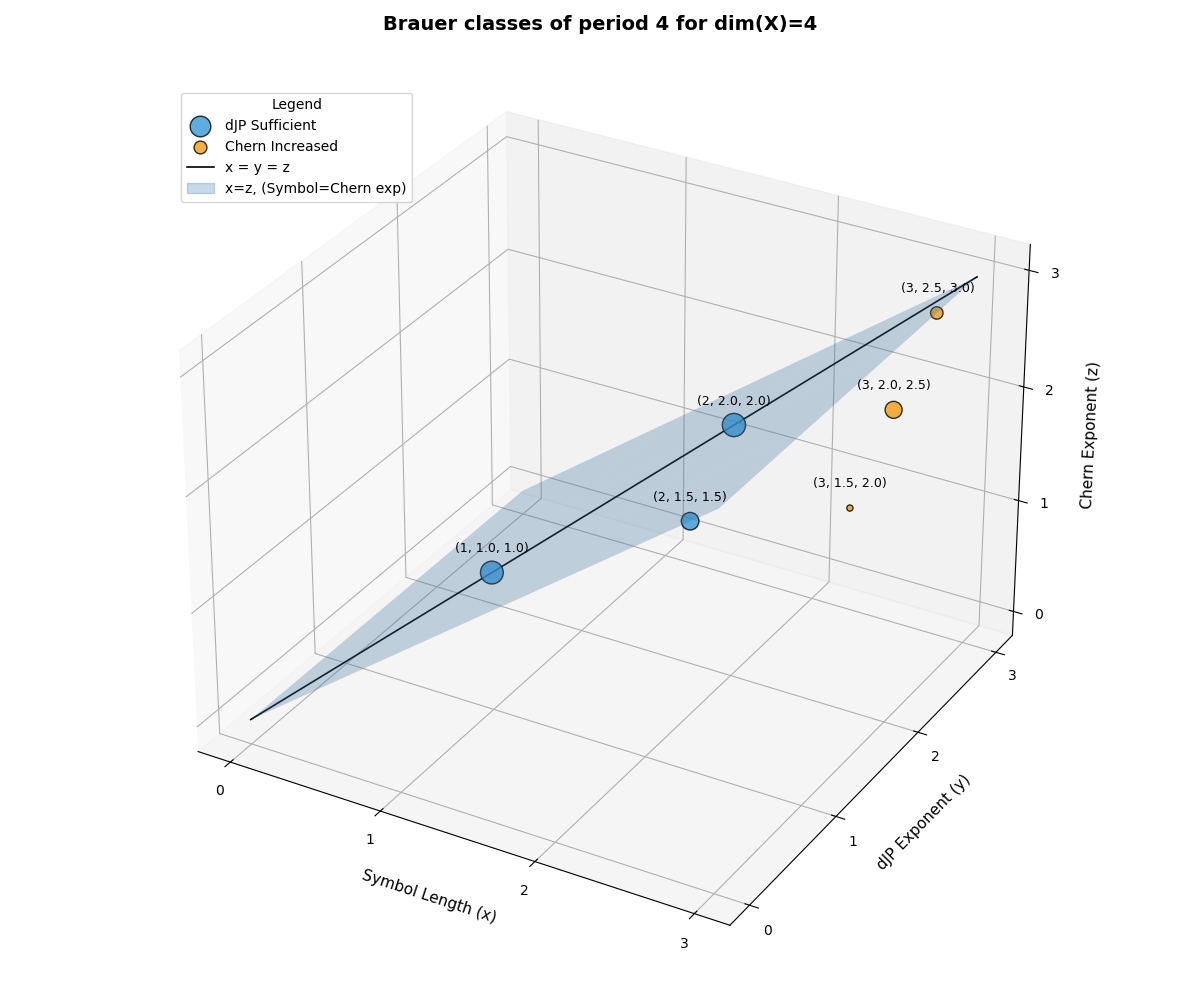}
    \end{adjustbox}
    \caption{A collection of 96 Brauer classes of period 4 on a very general abelian 4-fold sampled uniformly at random from 2-forms of Hamming weight at most 6. Blue bubbles indicate that either the maximum possible index (determined from symbol length) was obtained by the de Jong--Perry obstructions, or the Chern obstructions provided no new bounds. Yellow bubbles indicate algebras where Chern obstructions provided higher index bounds over the de Jong--Perry obstructions.}
    \label{fig: bfield_data_4_4}
\end{figure}

\begin{figure}[h]
    \begin{adjustbox}{center}
    \includegraphics[width=1.75\textwidth]{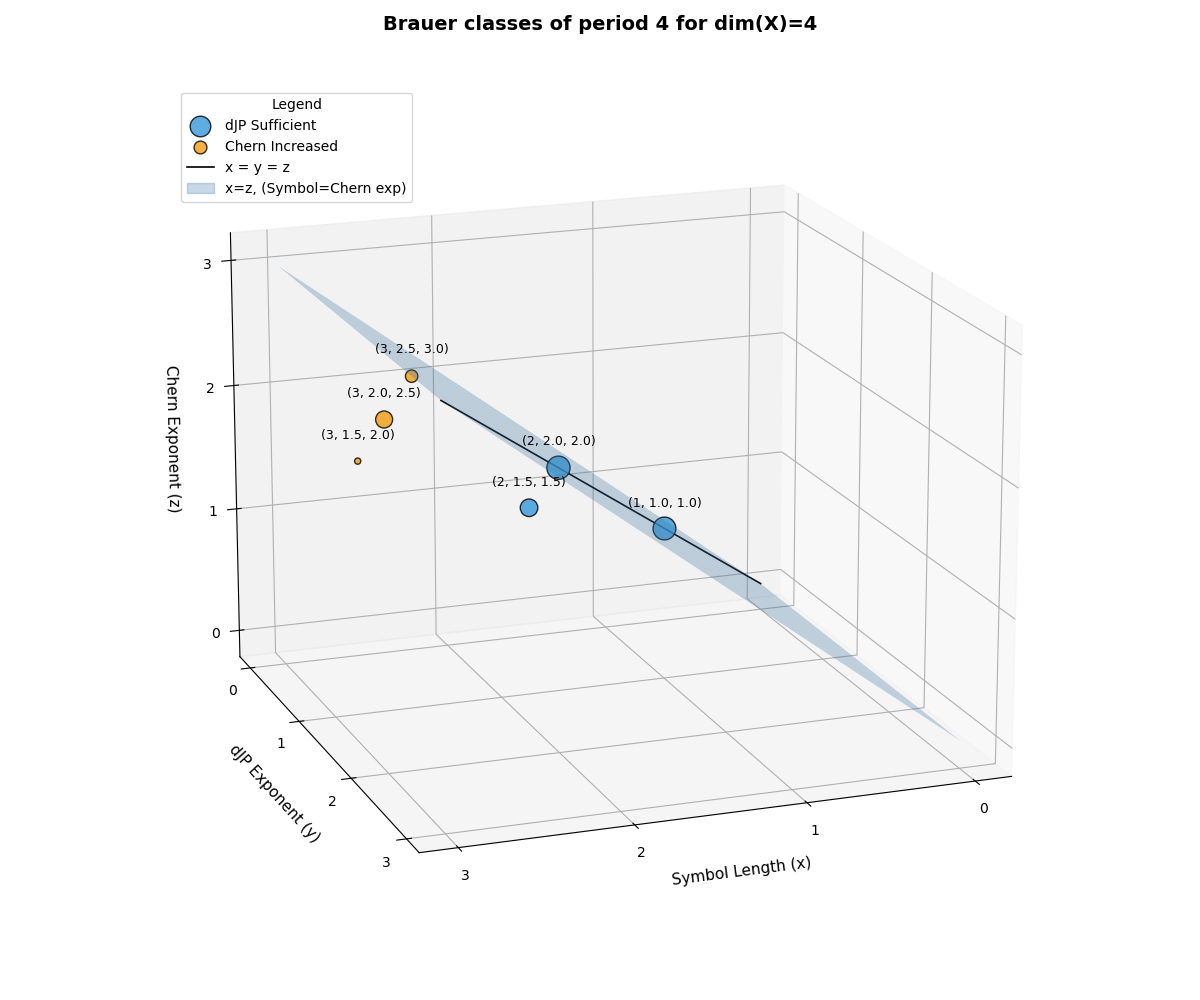}
    \end{adjustbox}
    \caption{A collection of 96 Brauer classes of period 4 on a very general abelian 4-fold sampled from 2-forms of Hamming weight at most 6, rotated.}
    \label{fig: bfield_data_4_4_r}
\end{figure}

\begin{figure}[h]
    \begin{adjustbox}{center}
    \includegraphics[width=1.6\textwidth]{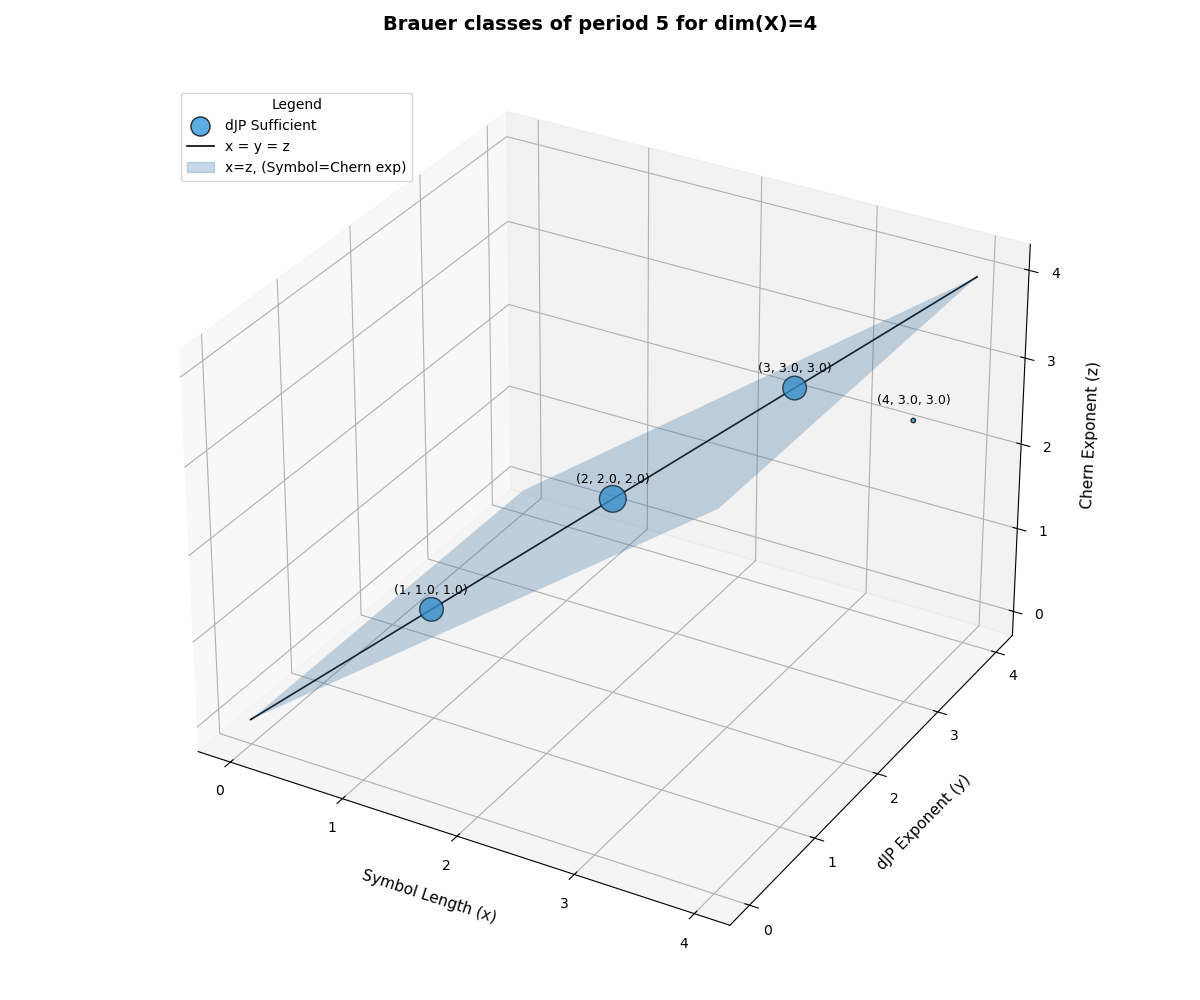}
    \end{adjustbox}
    \caption{A collection of 96 Brauer classes of period 5 on a very general abelian 4-fold sampled uniformly at random from 2-forms of Hamming weight at most 6. Blue bubbles indicate that either the maximum possible index (determined from symbol length) was obtained by the de Jong--Perry obstructions, or the Chern obstructions provided no new bounds. Yellow bubbles indicate algebras where Chern obstructions provided higher index bounds over the de Jong--Perry obstructions.}
    \label{fig: bfield_data_4_5}
\end{figure}

\begin{figure}[h]
    \begin{adjustbox}{center}
    \includegraphics[width=1.75\textwidth]{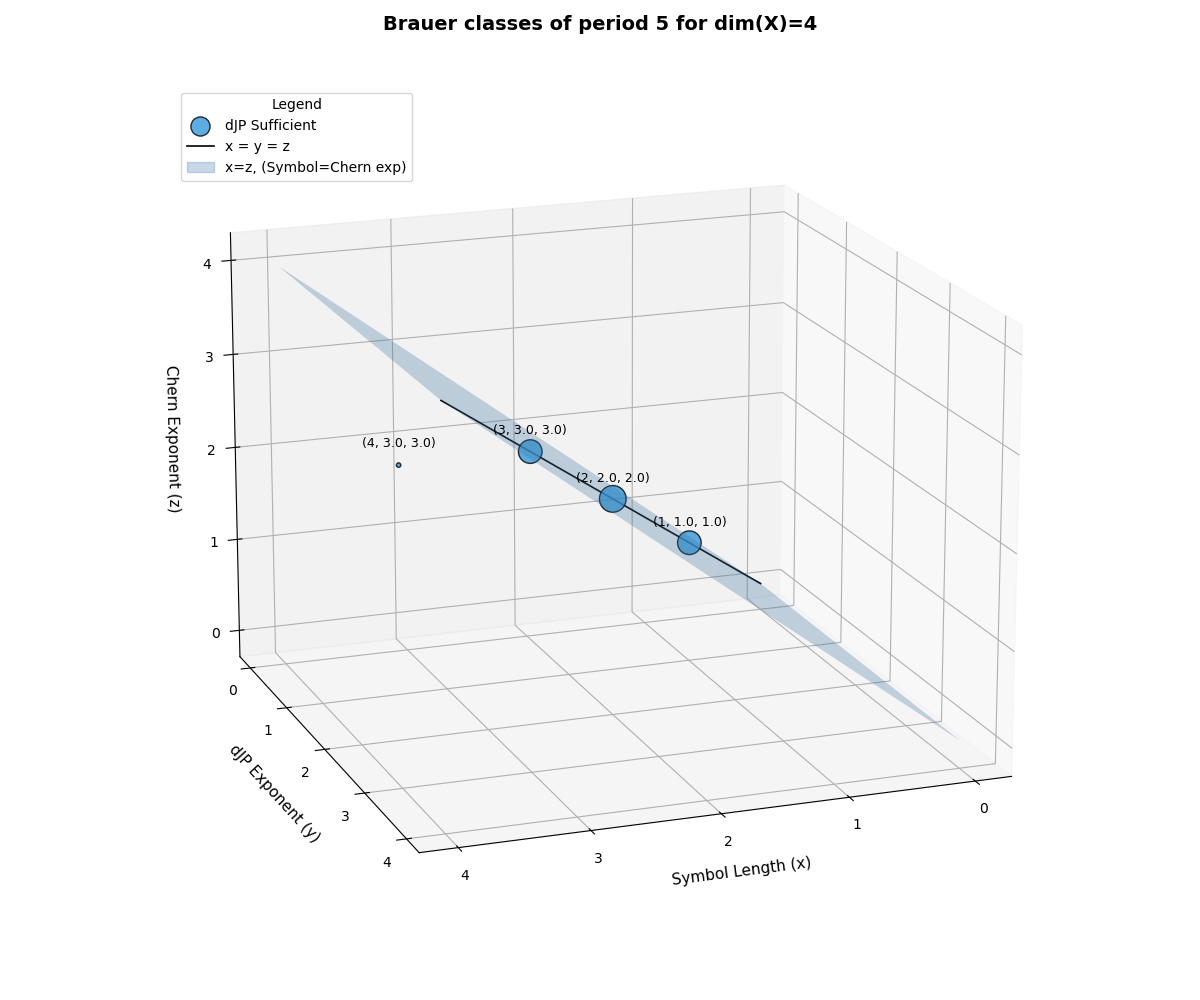}
    \end{adjustbox}
    \caption{A collection of 96 Brauer classes of period 5 on a very general abelian 4-fold sampled from 2-forms of Hamming weight at most 6, rotated.}
    \label{fig: bfield_data_4_5_r}
\end{figure}

\begin{figure}[h]
    \begin{adjustbox}{center}
    \includegraphics[width=1.6\textwidth]{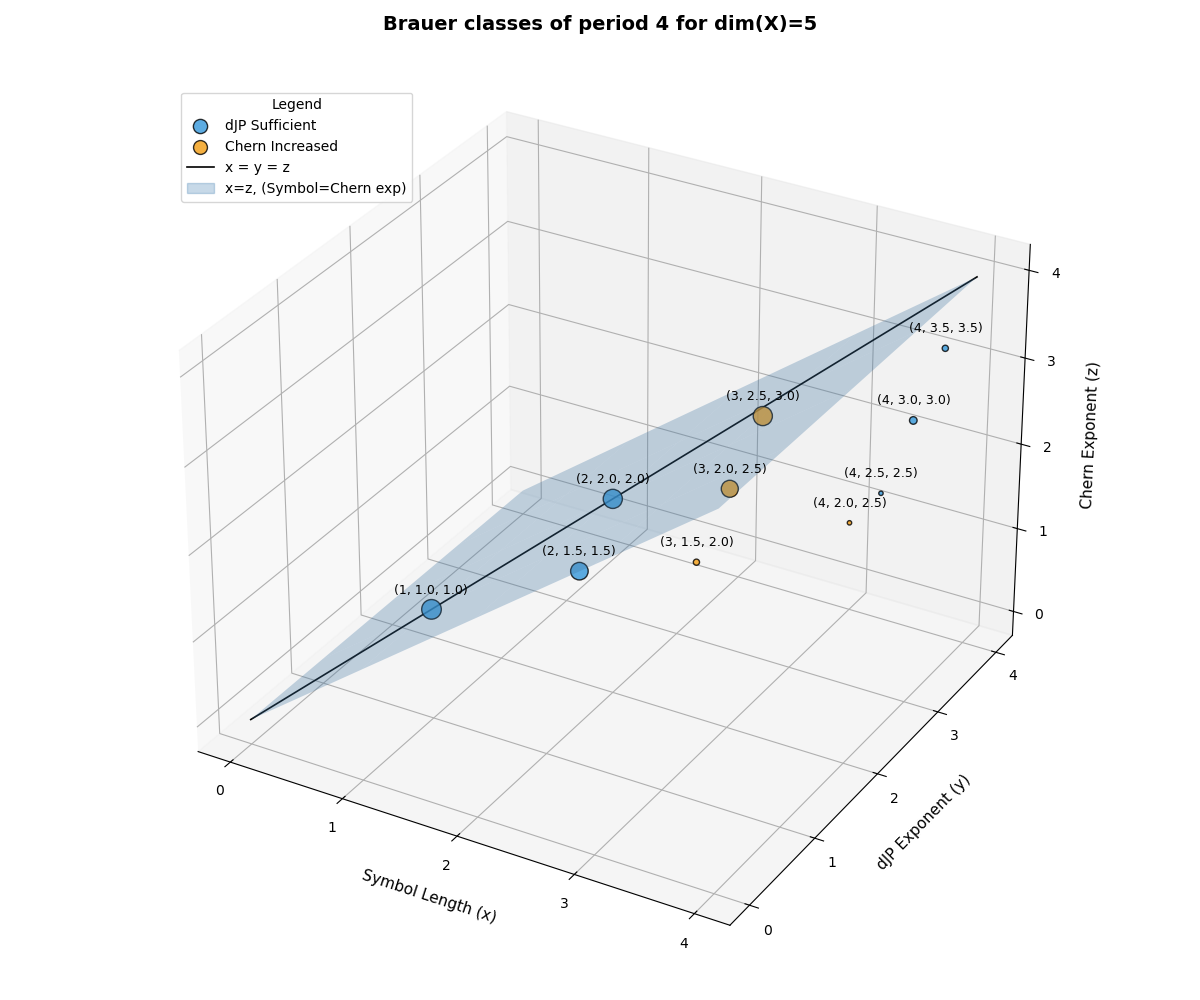}
    \end{adjustbox}
    \caption{A collection of 98 Brauer classes of period 4 on a very general abelian 5-fold sampled uniformly at random from 2-forms of Hamming weight at most 7. Blue bubbles indicate that either the maximum possible index (determined from symbol length) was obtained by the de Jong--Perry obstructions, or the Chern obstructions provided no new bounds. Yellow bubbles indicate algebras where Chern obstructions provided higher index bounds over the de Jong--Perry obstructions.}
    \label{fig: bfield_data_5_4}
\end{figure}

\begin{figure}[h]
    \begin{adjustbox}{center}
    \includegraphics[width=1.75\textwidth]{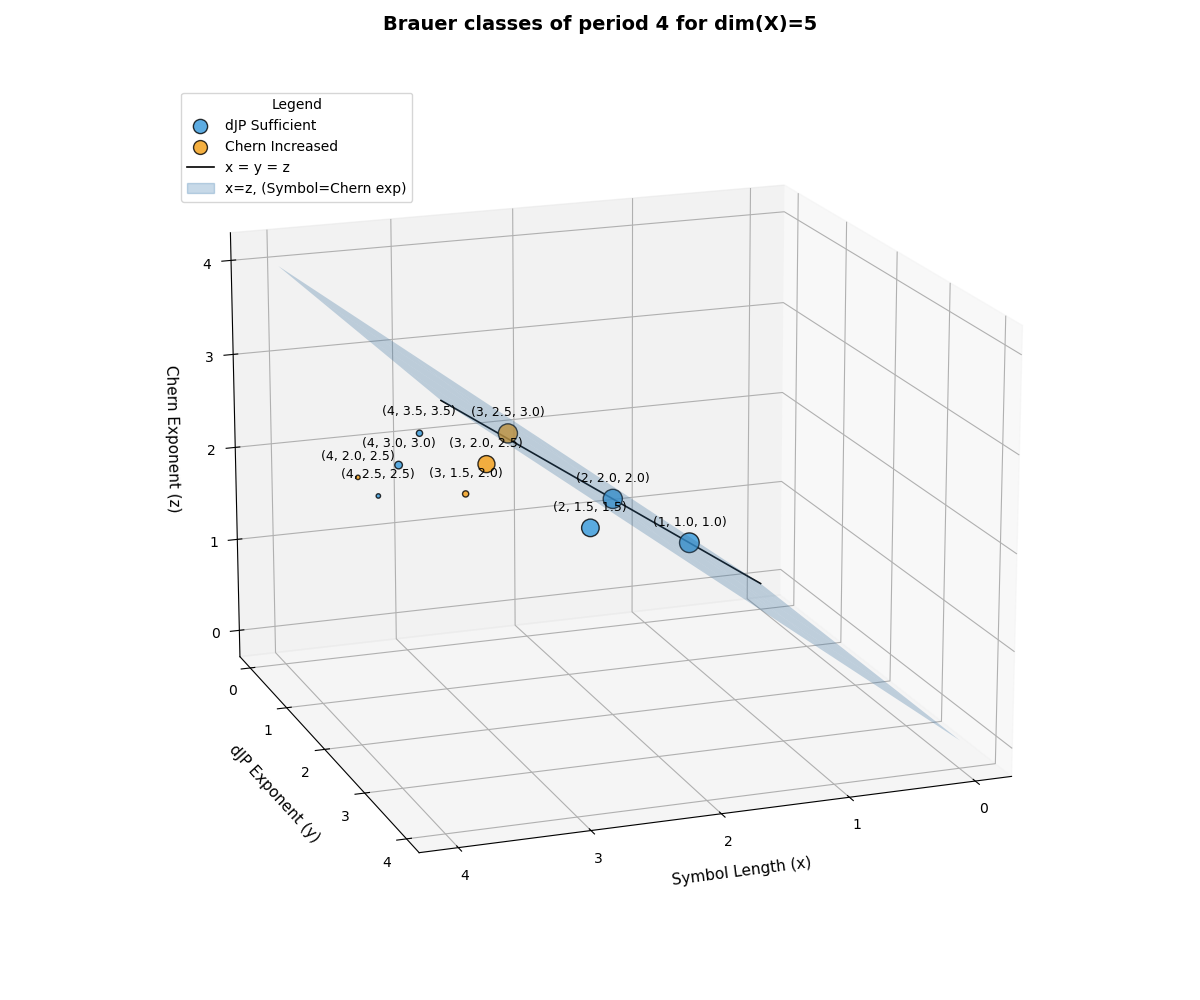}
    \end{adjustbox}
    \caption{A collection of 98 Brauer classes of period 4 on a very general abelian 5-fold sampled from 2-forms of Hamming weight at most 7, rotated.}
    \label{fig: bfield_data_5_4_r}
\end{figure}

\begin{figure}[h]
    \begin{adjustbox}{center}
    \includegraphics[width=1.6\textwidth]{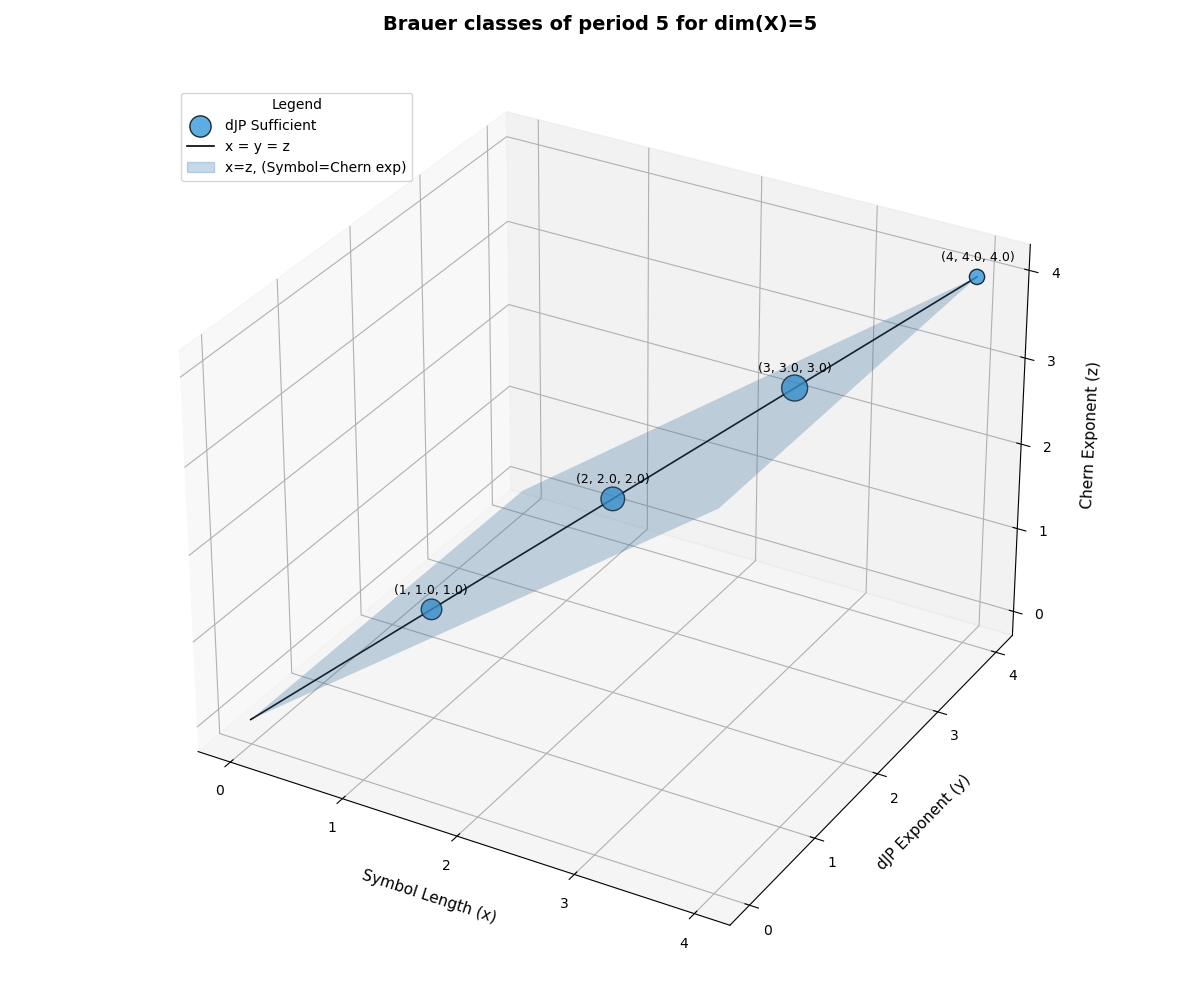}
    \end{adjustbox}
    \caption{A collection of 98 Brauer classes of period 5 on a very general abelian 5-fold sampled uniformly at random from 2-forms of Hamming weight at most 7. Blue bubbles indicate that either the maximum possible index (determined from symbol length) was obtained by the de Jong--Perry obstructions, or the Chern obstructions provided no new bounds. Yellow bubbles indicate algebras where Chern obstructions provided higher index bounds over the de Jong--Perry obstructions.}
    \label{fig: bfield_data_5_5}
\end{figure}

\begin{figure}[h]
    \begin{adjustbox}{center}
    \includegraphics[width=1.75\textwidth]{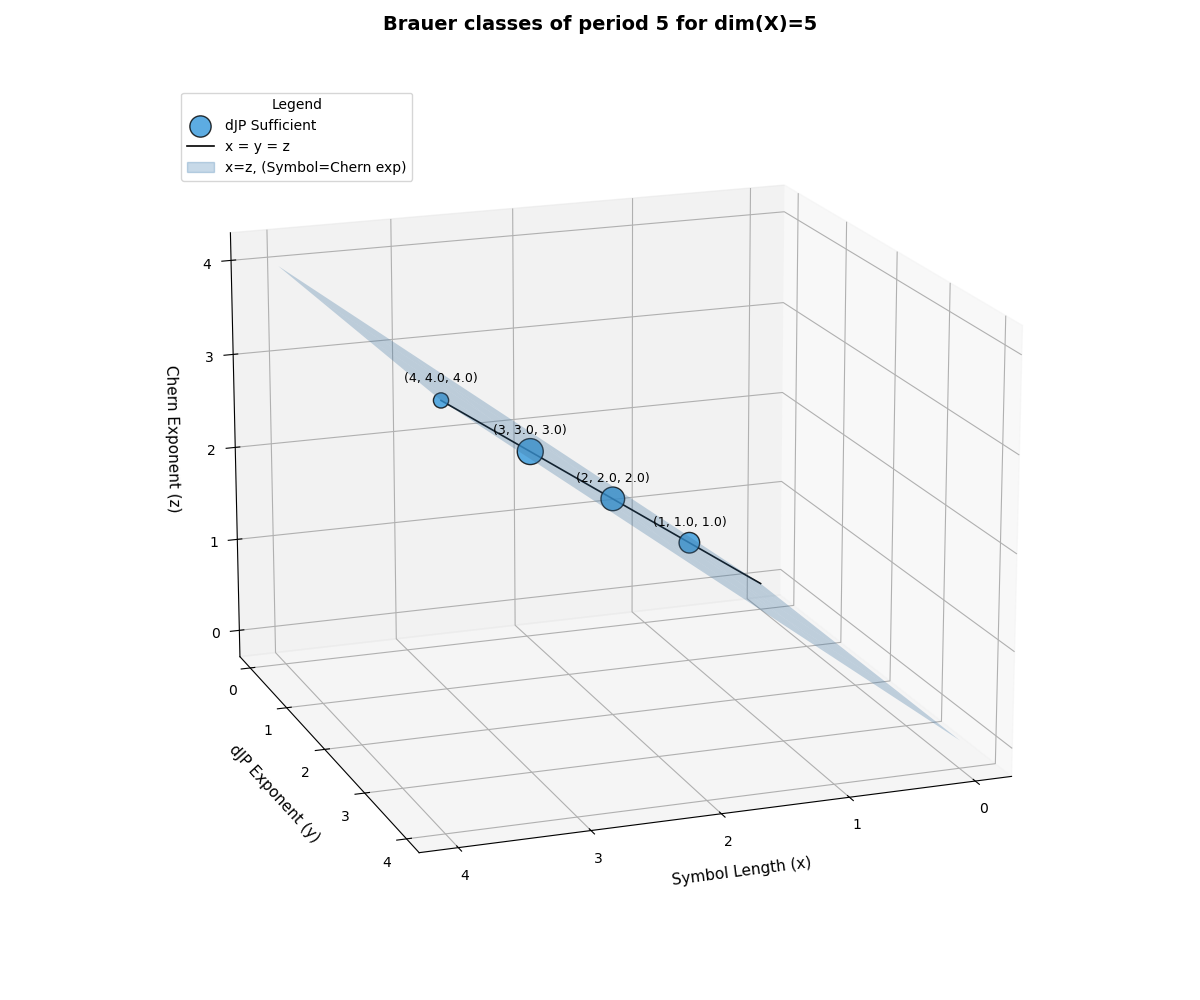}
    \end{adjustbox}
    \caption{A collection of 98 Brauer classes of period 5 on a very general abelian 5-fold sampled from 2-forms of Hamming weight at most 7, rotated.}
    \label{fig: bfield_data_5_5_r}
\end{figure}

\begin{figure}[h]
    \begin{adjustbox}{center}
    \includegraphics[width=1.6\textwidth]{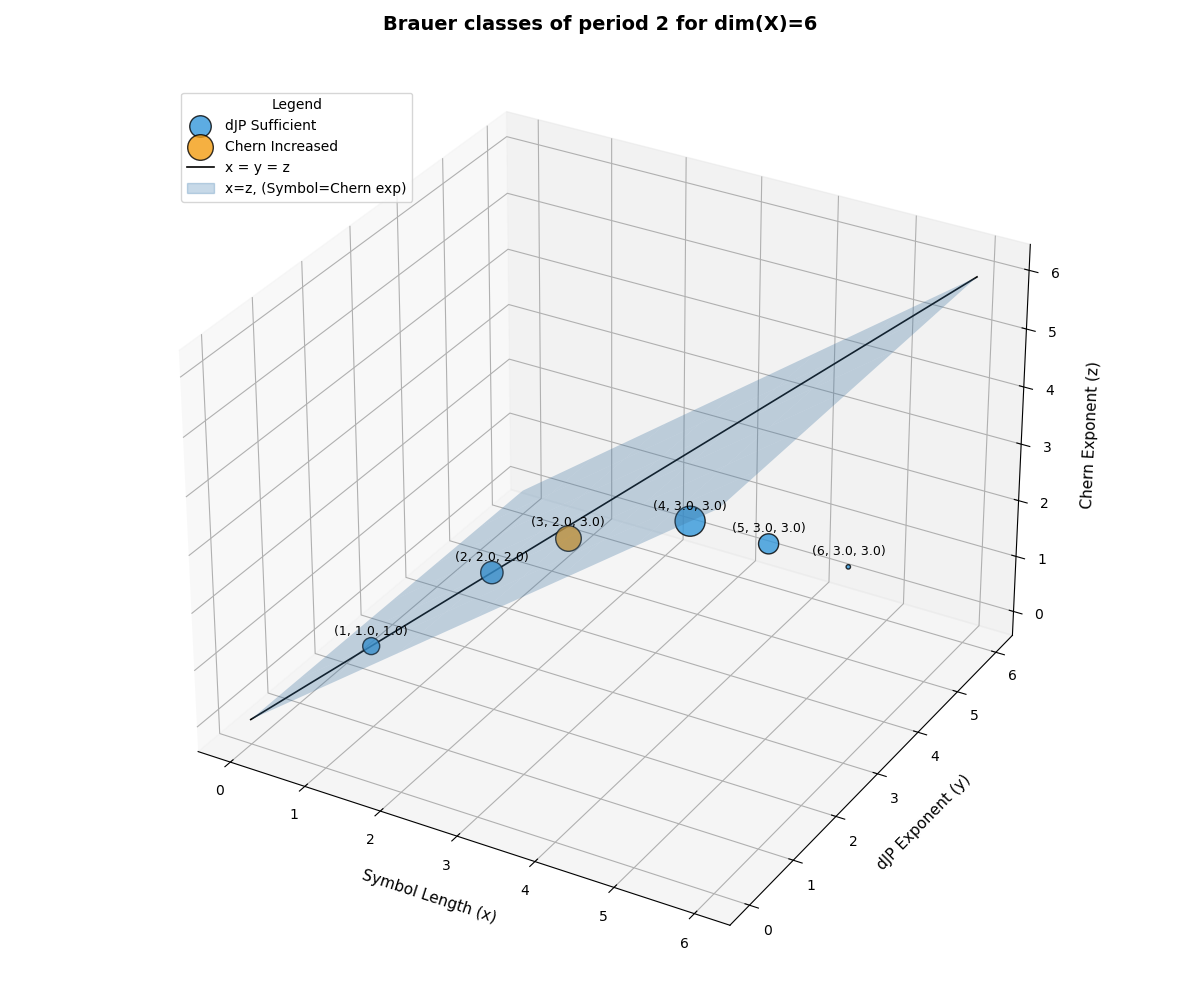}
    \end{adjustbox}
    \caption{A collection of 144 Brauer classes of period 2 on a very general abelian 6-fold sampled uniformly at random from 2-forms of Hamming weight at most 12. Blue bubbles indicate that either the maximum possible index (determined from symbol length) was obtained by the de Jong--Perry obstructions, or the Chern obstructions provided no new bounds. Yellow bubbles indicate algebras where Chern obstructions provided higher index bounds over the de Jong--Perry obstructions.}
    \label{fig: bfield_data_6_2_combined}
\end{figure}

\begin{figure}[h]
    \begin{adjustbox}{center}
    \includegraphics[width=1.75\textwidth]{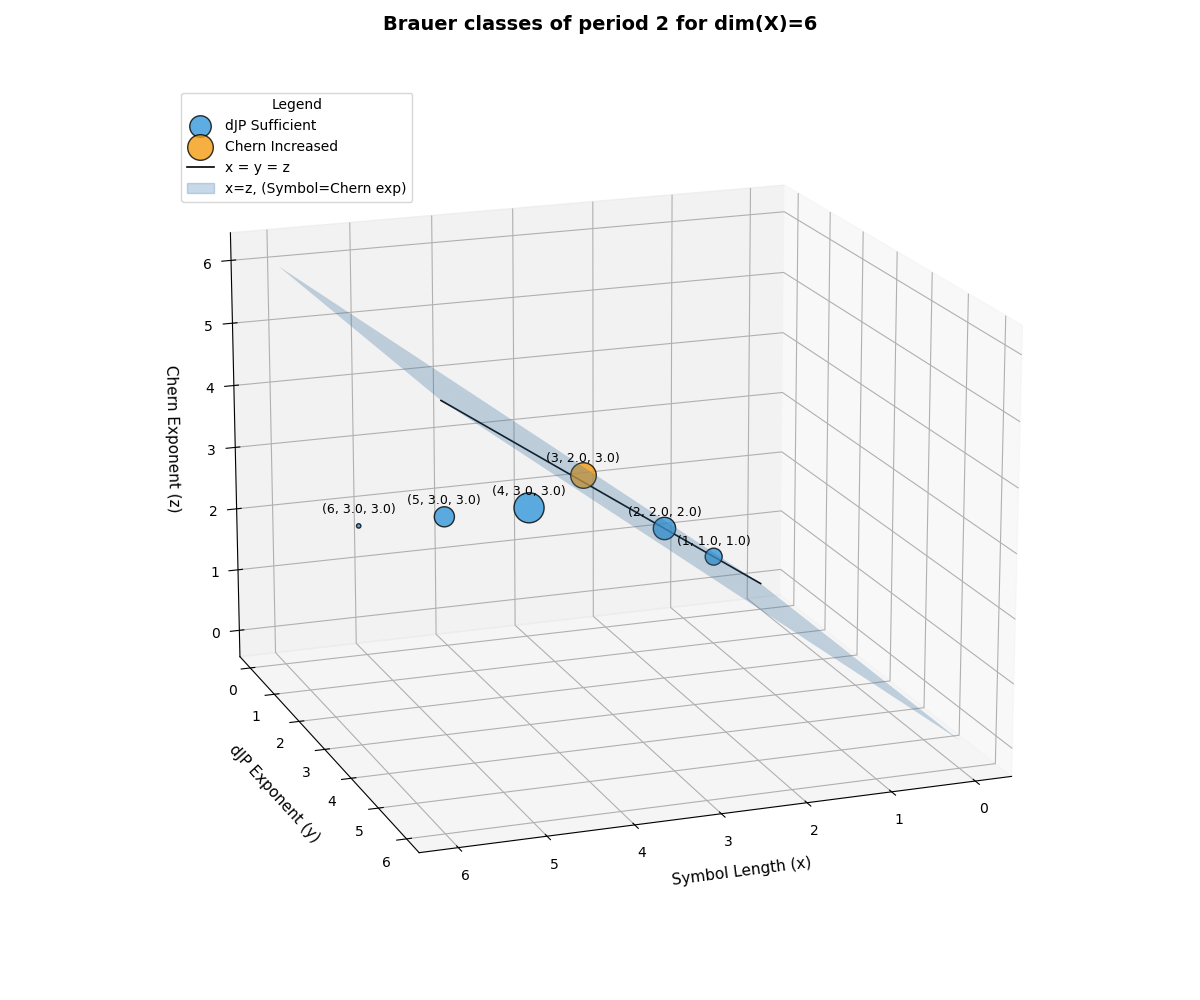}
    \end{adjustbox}
    \caption{A collection of 144 Brauer classes of period 2 on a very general abelian 6-fold sampled from 2-forms of Hamming weight at most 12, rotated.}
    \label{fig: bfield_data_6_2_combined_r}
\end{figure}

\begin{figure}[h]
    \begin{adjustbox}{center}
    \includegraphics[width=1.6\textwidth]{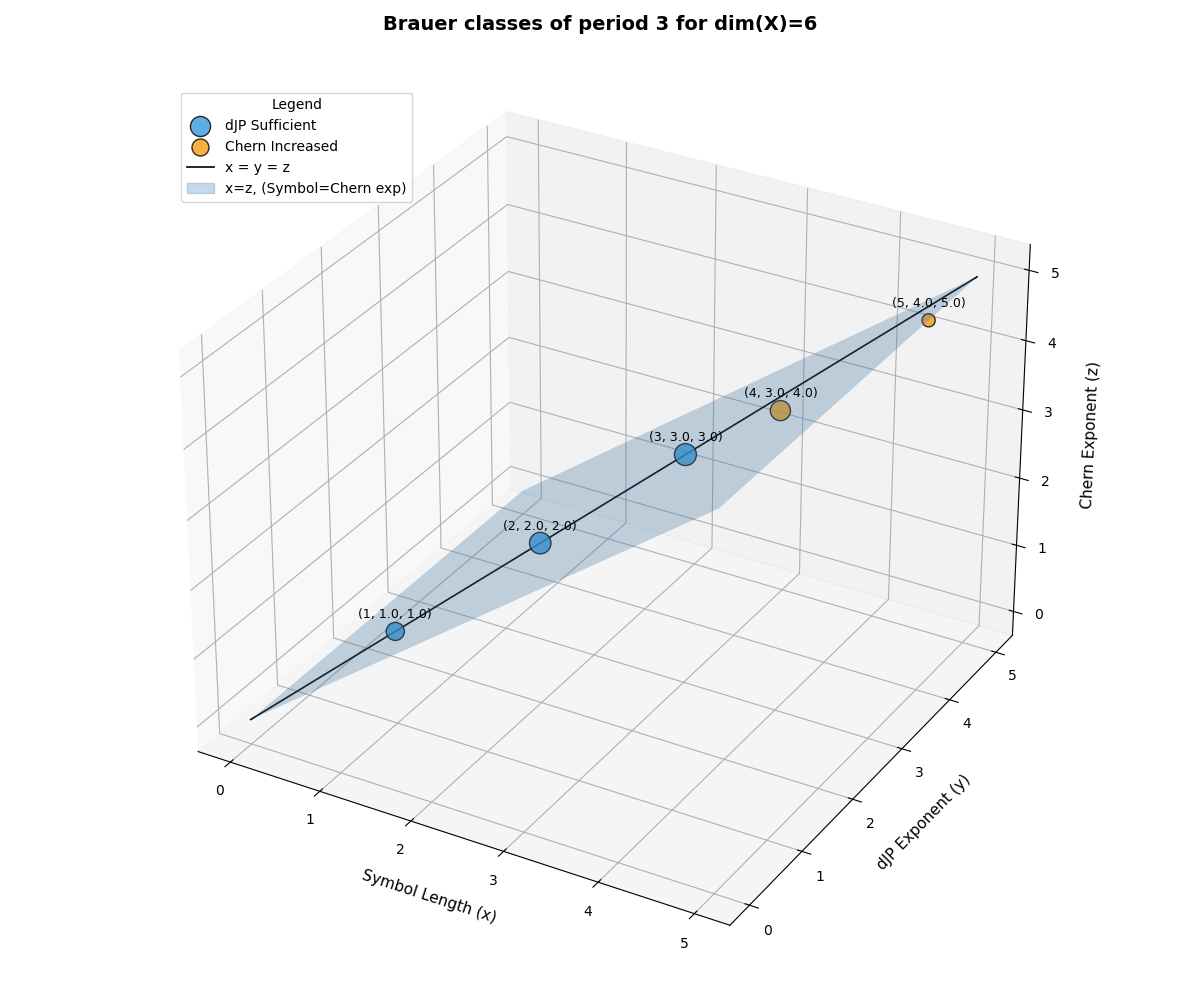}
    \end{adjustbox}
    \caption{A collection of 96 Brauer classes of period 3 on a very general abelian 6-fold sampled uniformly at random from 2-forms of Hamming weight at most 8. Blue bubbles indicate that either the maximum possible index (determined from symbol length) was obtained by the de Jong--Perry obstructions, or the Chern obstructions provided no new bounds. Yellow bubbles indicate algebras where Chern obstructions provided higher index bounds over the de Jong--Perry obstructions.}
    \label{fig: bfield_data_6_3}
\end{figure}

\begin{figure}[h]
    \begin{adjustbox}{center}
    \includegraphics[width=1.75\textwidth]{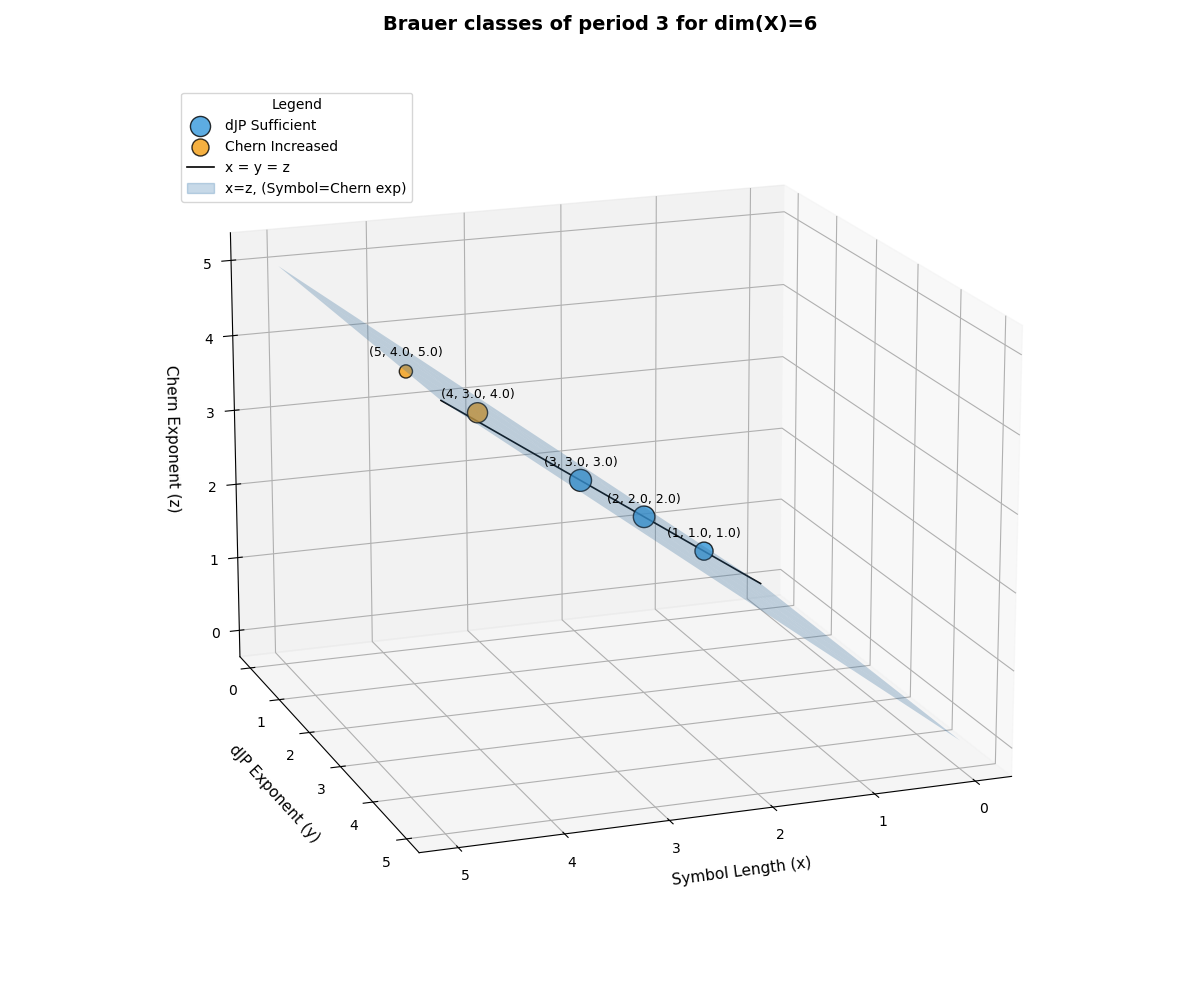}
    \end{adjustbox}
    \caption{A collection of 96 Brauer classes of period 3 on a very general abelian 6-fold sampled from 2-forms of Hamming weight at most 8, rotated.}
    \label{fig: bfield_data_6_3_r}
\end{figure}

\begin{figure}[h]
    \begin{adjustbox}{center}
    \includegraphics[width=1.6\textwidth]{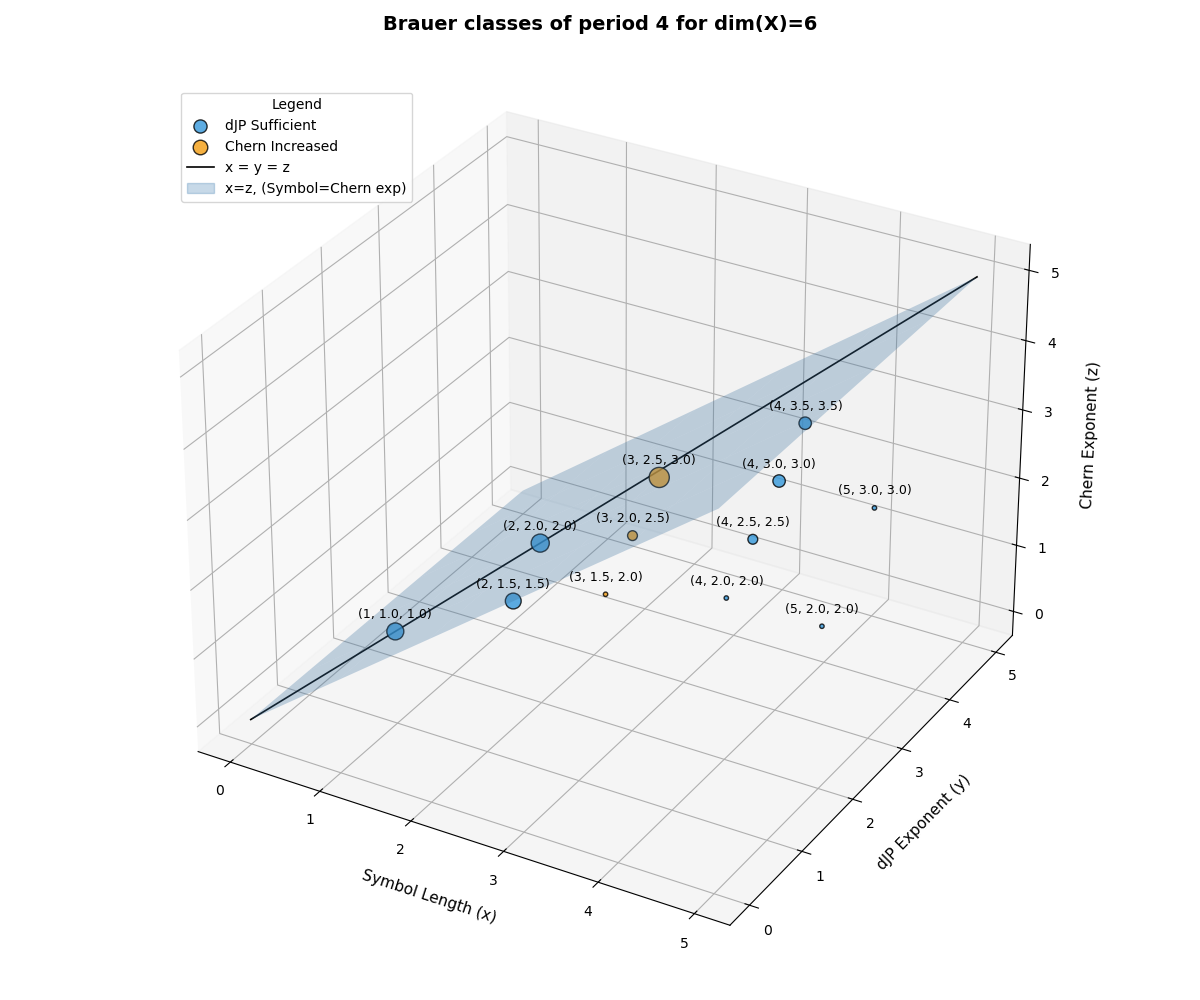}
    \end{adjustbox}
    \caption{A collection of 96 Brauer classes of period 4 on a very general abelian 6-fold sampled uniformly at random from 2-forms of Hamming weight at most 8. Blue bubbles indicate that either the maximum possible index (determined from symbol length) was obtained by the de Jong--Perry obstructions, or the Chern obstructions provided no new bounds. Yellow bubbles indicate algebras where Chern obstructions provided higher index bounds over the de Jong--Perry obstructions.}
    \label{fig: bfield_data_6_4}
\end{figure}

\begin{figure}[h]
    \begin{adjustbox}{center}
    \includegraphics[width=1.75\textwidth]{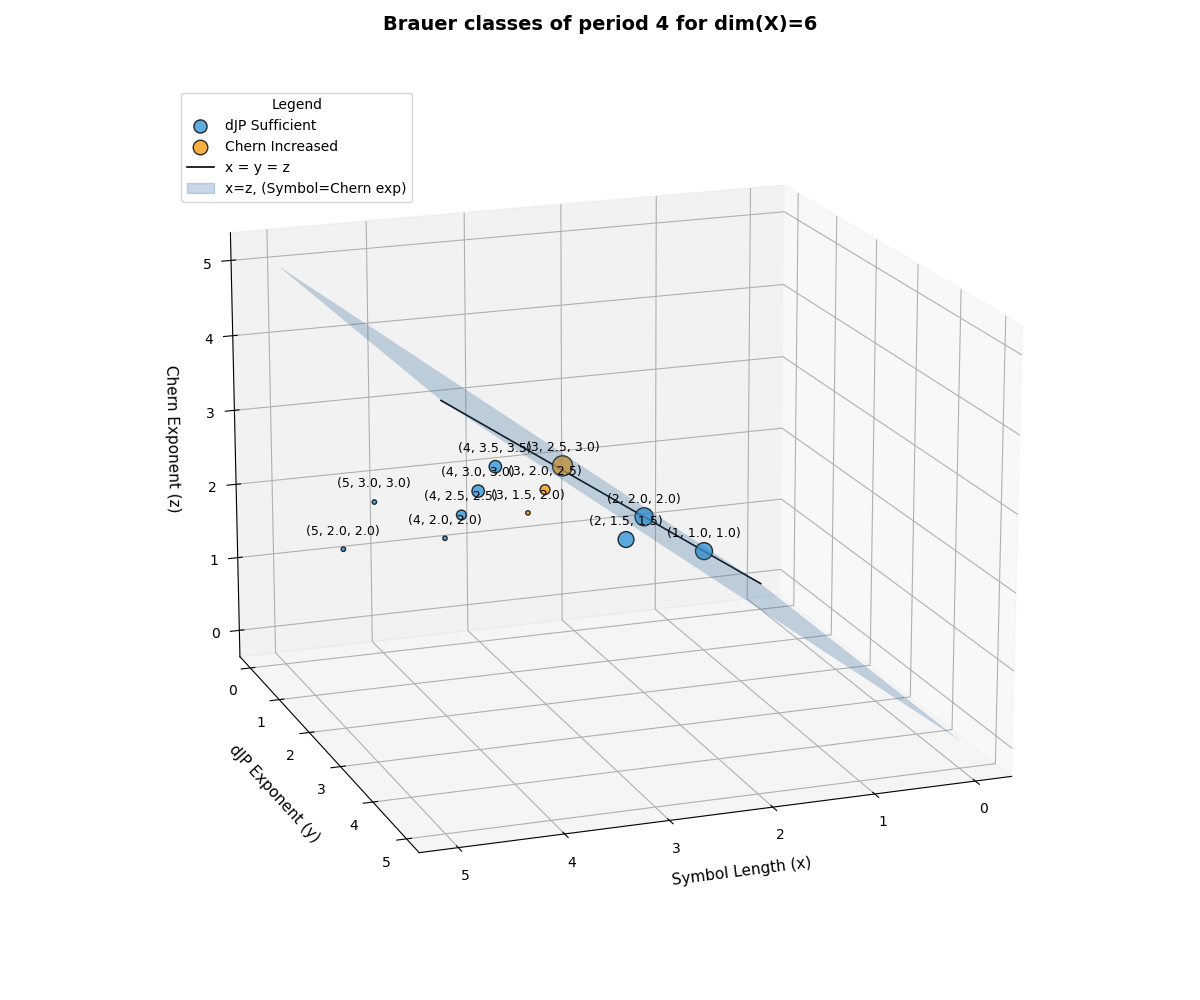}
    \end{adjustbox}
    \caption{A collection of 96 Brauer classes of period 4 on a very general abelian 6-fold sampled from 2-forms of Hamming weight at most 8, rotated.}
    \label{fig: bfield_data_6_4_r}
\end{figure}

\begin{figure}[h]
    \begin{adjustbox}{center}
    \includegraphics[width=1.6\textwidth]{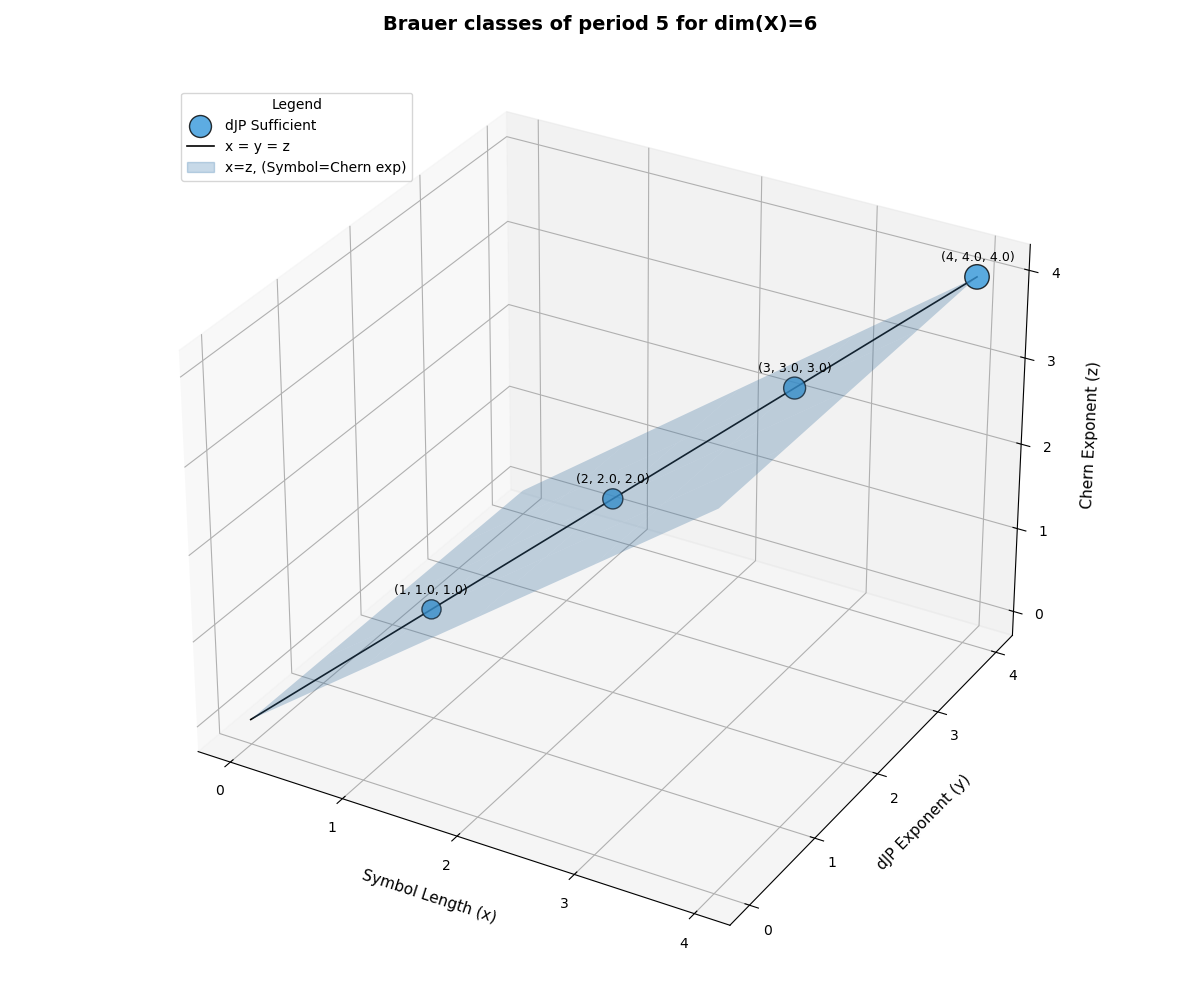}
    \end{adjustbox}
    \caption{A collection of 96 Brauer classes of period 5 on a very general abelian 6-fold sampled uniformly at random from 2-forms of Hamming weight at most 8. Blue bubbles indicate that either the maximum possible index (determined from symbol length) was obtained by the de Jong--Perry obstructions, or the Chern obstructions provided no new bounds. Yellow bubbles indicate algebras where Chern obstructions provided higher index bounds over the de Jong--Perry obstructions.}
    \label{fig: bfield_data_6_5}
\end{figure}

\begin{figure}[h]
    \begin{adjustbox}{center}
    \includegraphics[width=1.75\textwidth]{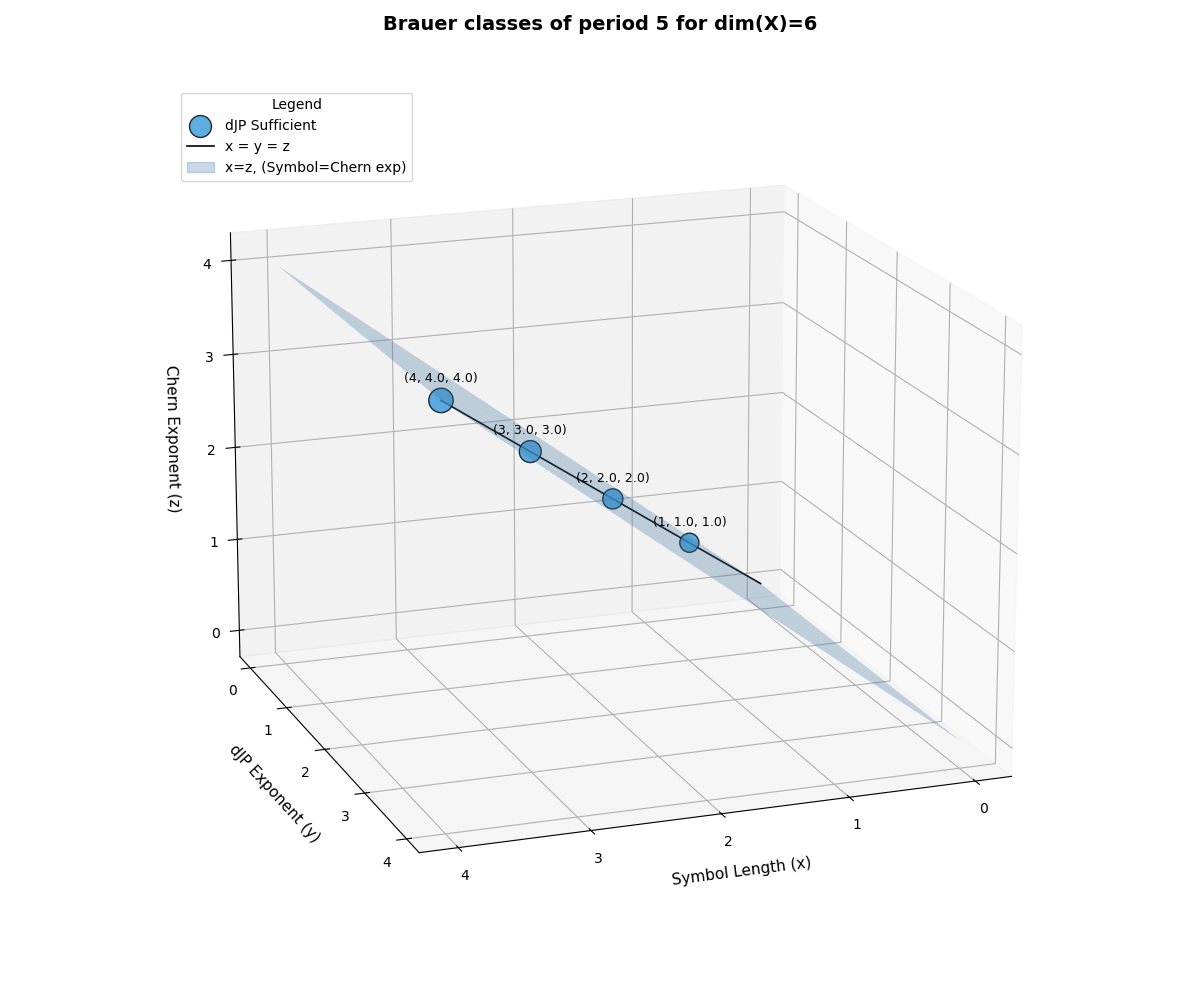}
    \end{adjustbox}
    \caption{A collection of 96 Brauer classes of period 5 on a very general abelian 6-fold sampled from 2-forms of Hamming weight at most 8, rotated.}
    \label{fig: bfield_data_6_5_r}
\end{figure}

\begin{figure}[h]
    \begin{adjustbox}{center}
    \includegraphics[width=1.6\textwidth]{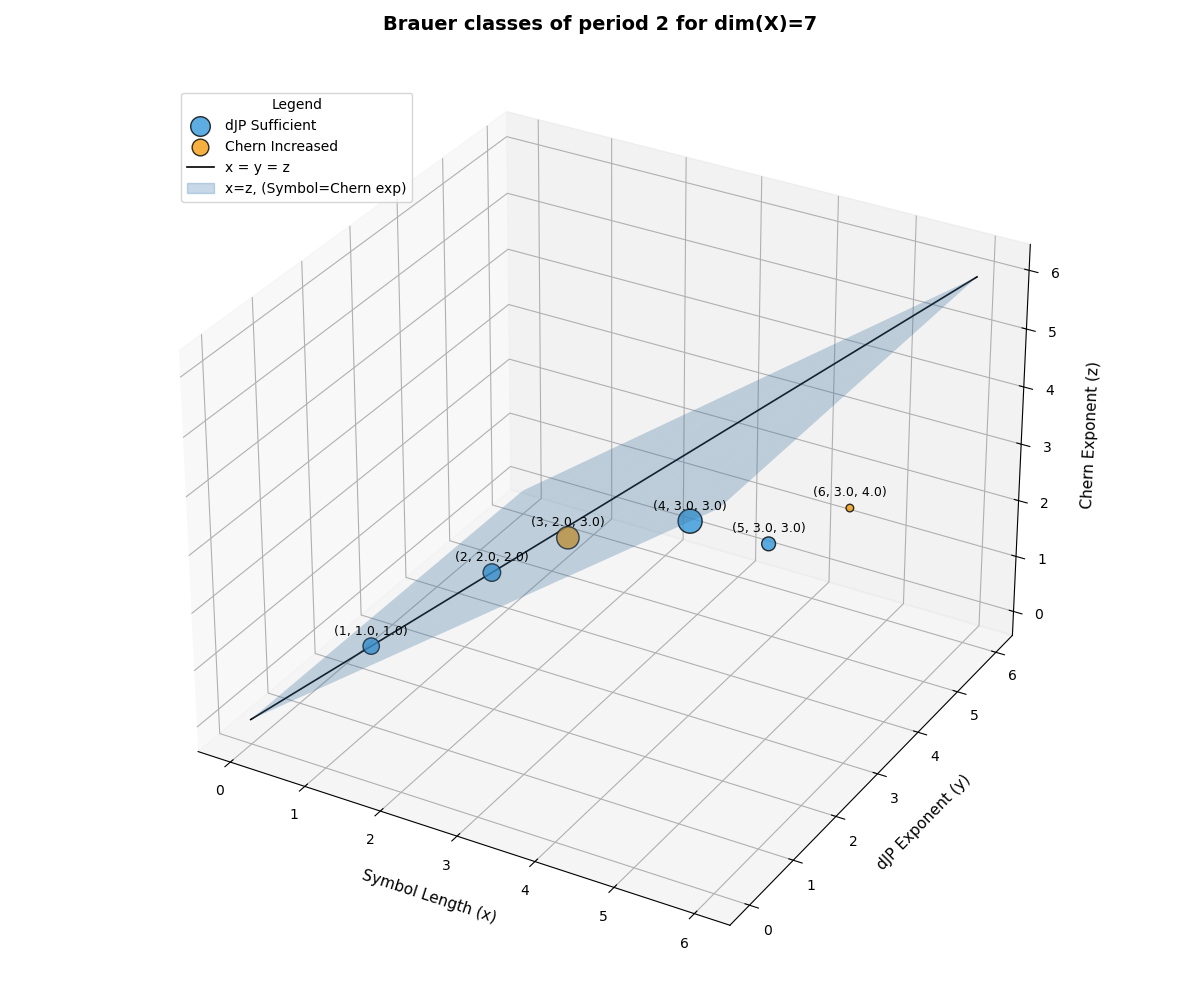}
    \end{adjustbox}
    \caption{A collection of 99 Brauer classes of period 2 on a very general abelian 7-fold sampled uniformly at random from 2-forms of Hamming weight at most 9. Blue bubbles indicate that either the maximum possible index (determined from symbol length) was obtained by the de Jong--Perry obstructions, or the Chern obstructions provided no new bounds. Yellow bubbles indicate algebras where Chern obstructions provided higher index bounds over the de Jong--Perry obstructions.}
    \label{fig: bfield_data_7_2}
\end{figure}

\begin{figure}[h]
    \begin{adjustbox}{center}
    \includegraphics[width=1.75\textwidth]{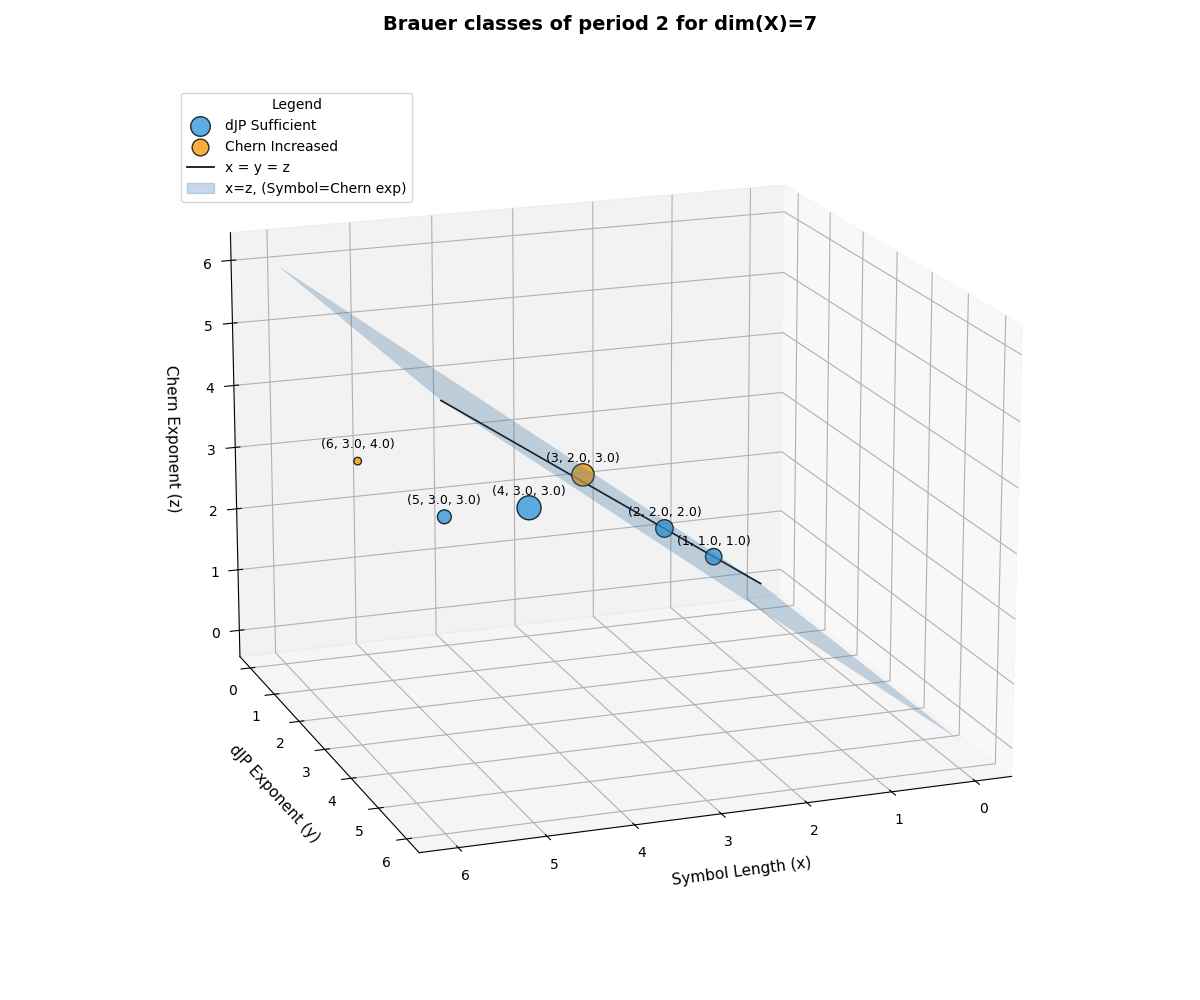}
    \end{adjustbox}
    \caption{A collection of 99 Brauer classes of period 2 on a very general abelian 7-fold sampled from 2-forms of Hamming weight at most 9, rotated.}
    \label{fig: bfield_data_7_2_r}
\end{figure}

\begin{figure}[h]
    \begin{adjustbox}{center}
    \includegraphics[width=1.6\textwidth]{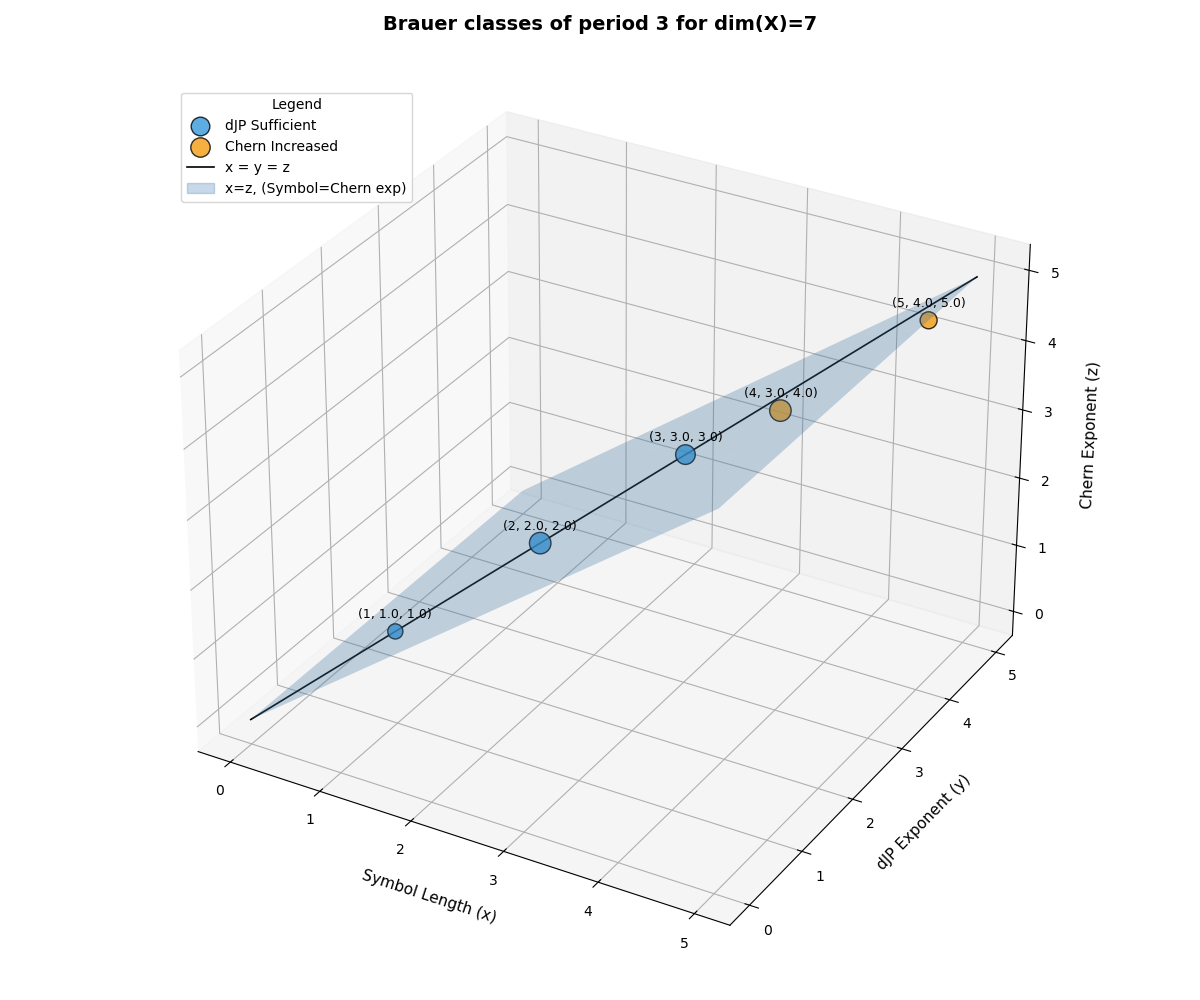}
    \end{adjustbox}
    \caption{A collection of 95 Brauer classes of period 3 on a very general abelian 7-fold sampled uniformly at random from 2-forms of Hamming weight at most 9. Blue bubbles indicate that either the maximum possible index (determined from symbol length) was obtained by the de Jong--Perry obstructions, or the Chern obstructions provided no new bounds. Yellow bubbles indicate algebras where Chern obstructions provided higher index bounds over the de Jong--Perry obstructions.}
    \label{fig: bfield_data_7_3}
\end{figure}

\begin{figure}[h]
    \begin{adjustbox}{center}
    \includegraphics[width=1.75\textwidth]{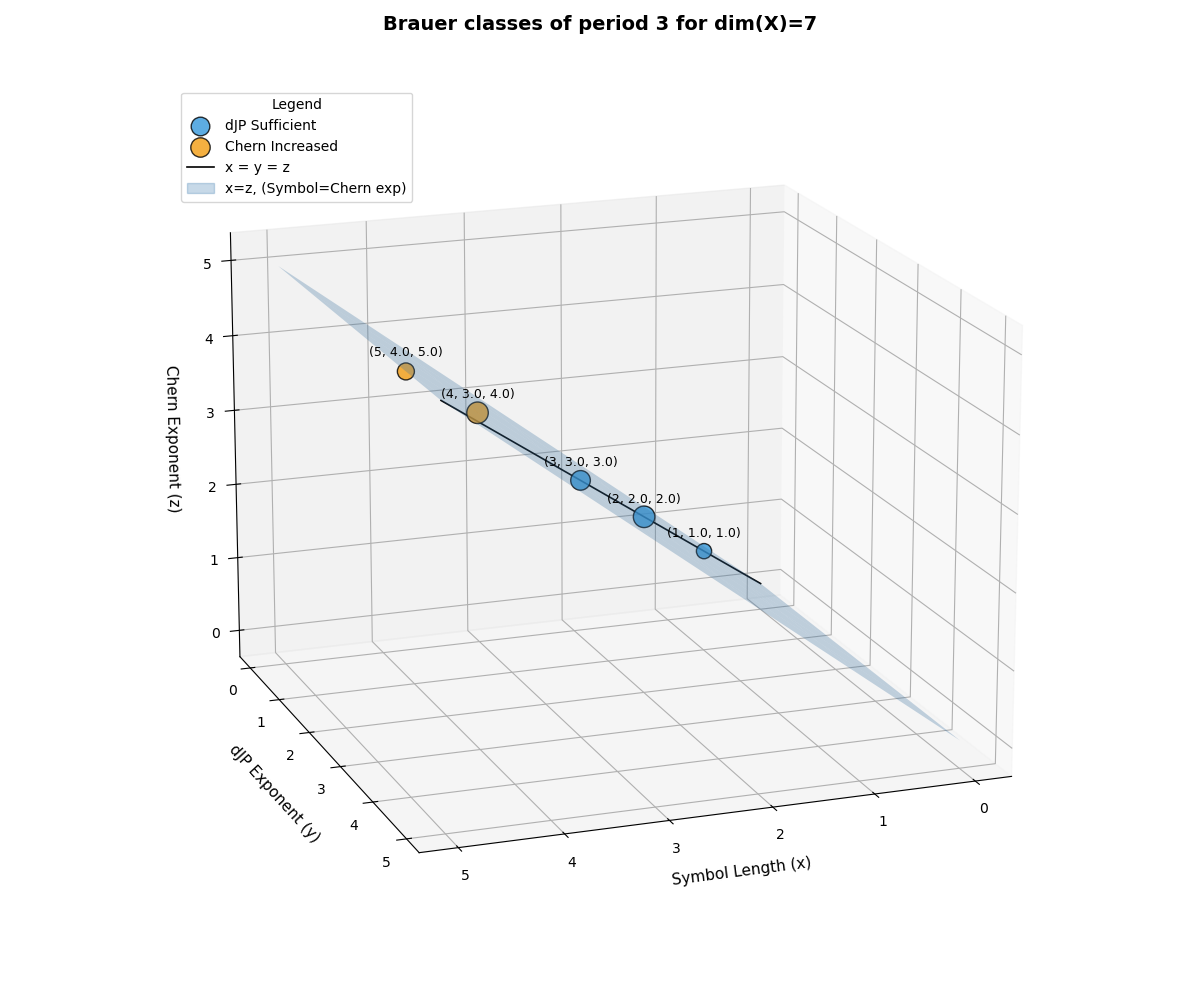}
    \end{adjustbox}
    \caption{A collection of 95 Brauer classes of period 3 on a very general abelian 7-fold sampled from 2-forms of Hamming weight at most 9, rotated.}
    \label{fig: bfield_data_7_3_r}
\end{figure}

\begin{figure}[h]
    \begin{adjustbox}{center}
    \includegraphics[width=1.6\textwidth]{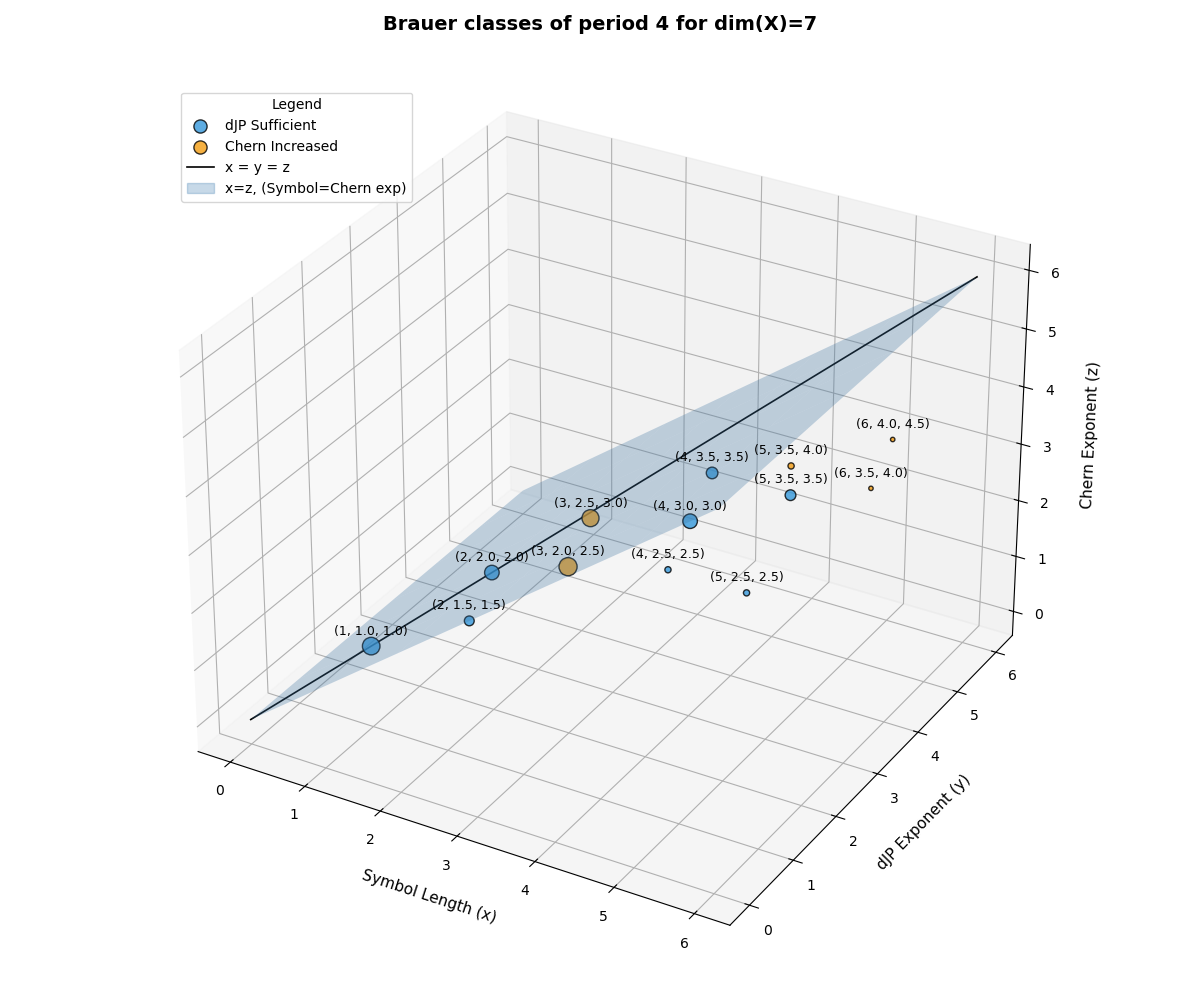}
    \end{adjustbox}
    \caption{A collection of 96 Brauer classes of period 4 on a very general abelian 7-fold sampled uniformly at random from 2-forms of Hamming weight at most 9. Blue bubbles indicate that either the maximum possible index (determined from symbol length) was obtained by the de Jong--Perry obstructions, or the Chern obstructions provided no new bounds. Yellow bubbles indicate algebras where Chern obstructions provided higher index bounds over the de Jong--Perry obstructions.}
    \label{fig: bfield_data_7_4}
\end{figure}

\begin{figure}[h]
    \begin{adjustbox}{center}
    \includegraphics[width=1.75\textwidth]{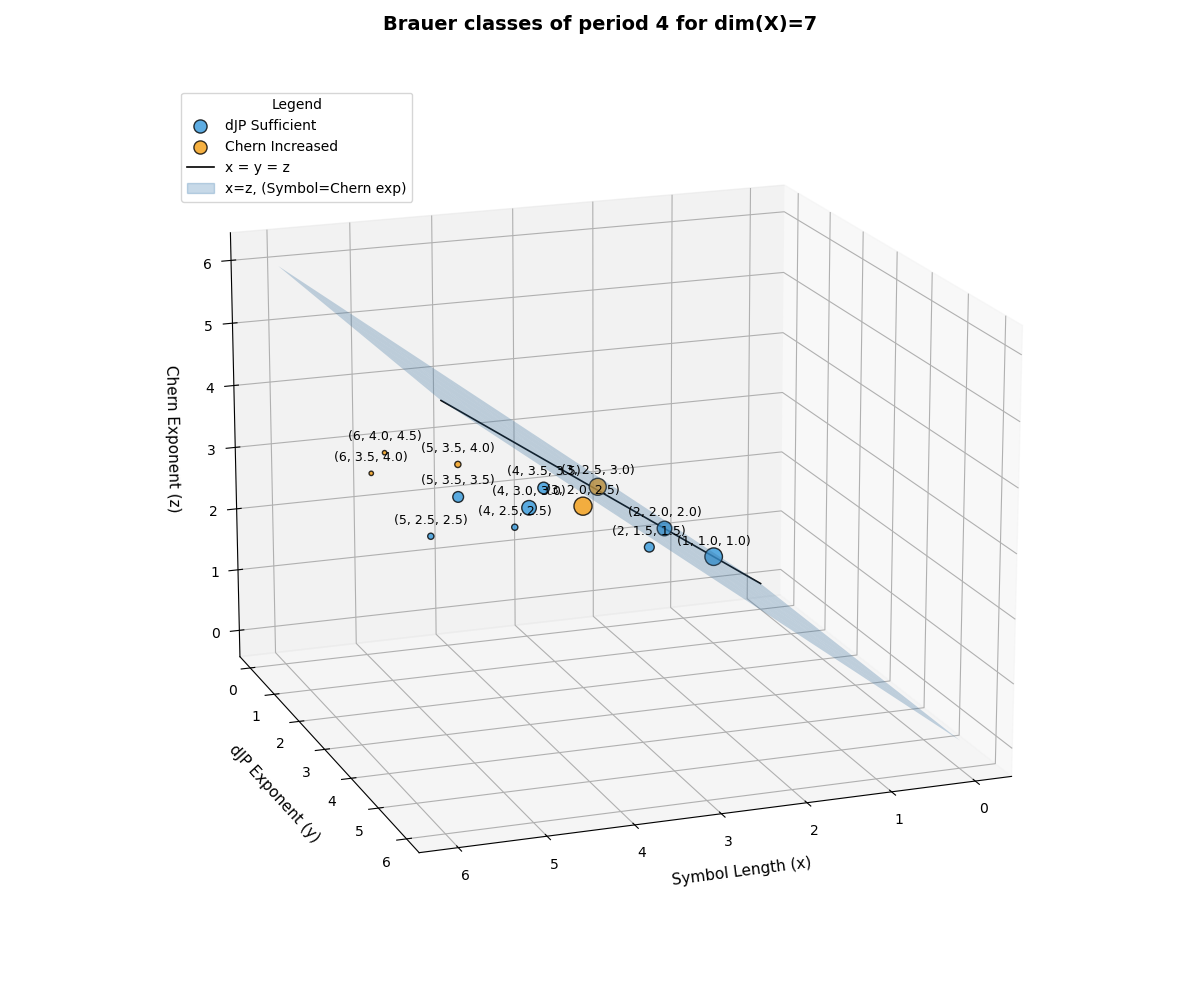}
    \end{adjustbox}
    \caption{A collection of 96 Brauer classes of period 4 on a very general abelian 7-fold sampled from 2-forms of Hamming weight at most 9, rotated.}
    \label{fig: bfield_data_7_4_r}
\end{figure}

\begin{figure}[h]
    \begin{adjustbox}{center}
    \includegraphics[width=1.6\textwidth]{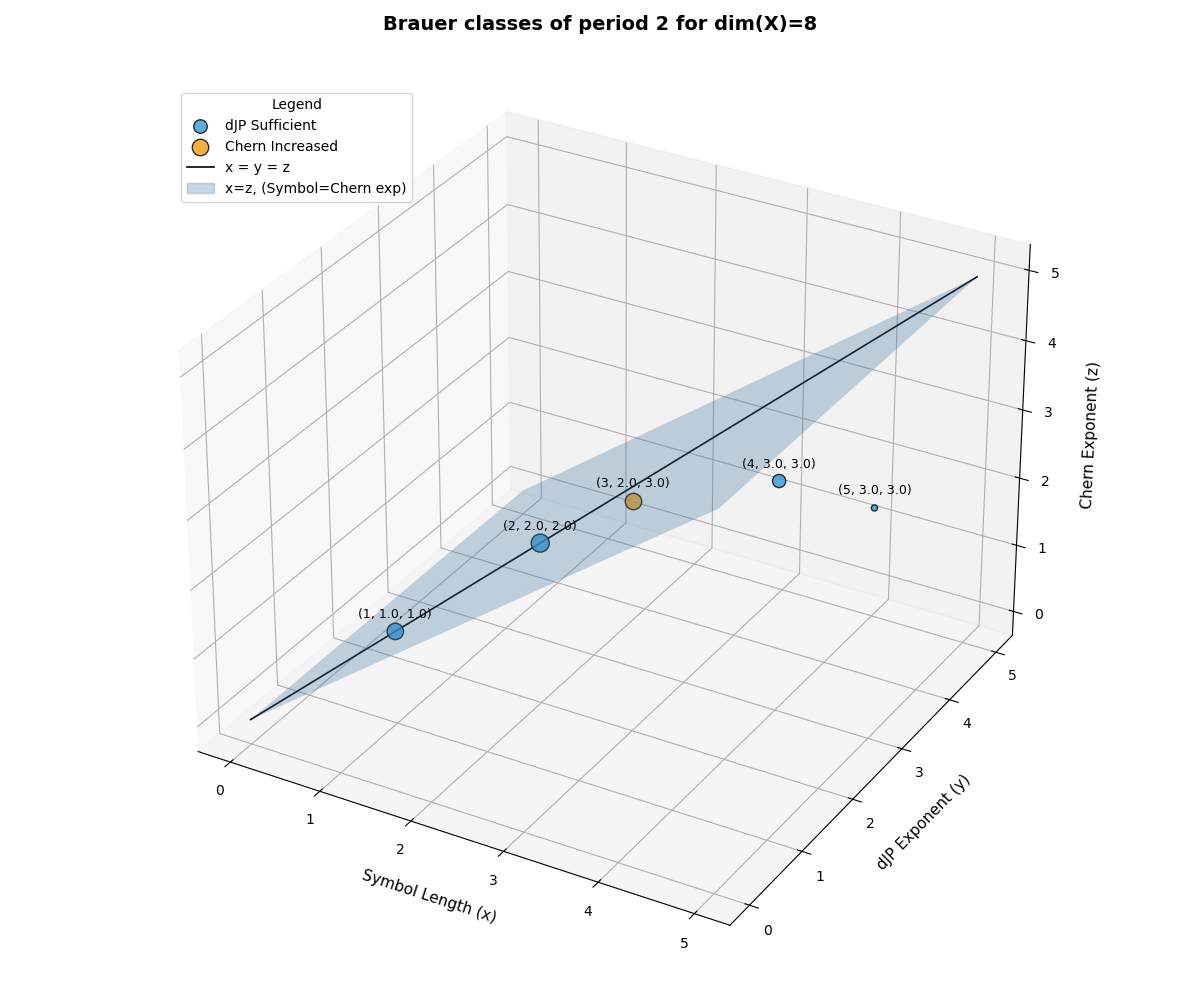}
    \end{adjustbox}
    \caption{A collection of 56 Brauer classes of period 2 on a very general abelian 8-fold sampled uniformly at random from 2-forms of Hamming weight at most 6. Blue bubbles indicate that either the maximum possible index (determined from symbol length) was obtained by the de Jong--Perry obstructions, or the Chern obstructions provided no new bounds. Yellow bubbles indicate algebras where Chern obstructions provided higher index bounds over the de Jong--Perry obstructions.}
    \label{fig: bfield_data_8_2}
\end{figure}

\begin{figure}[h]
    \begin{adjustbox}{center}
    \includegraphics[width=1.75\textwidth]{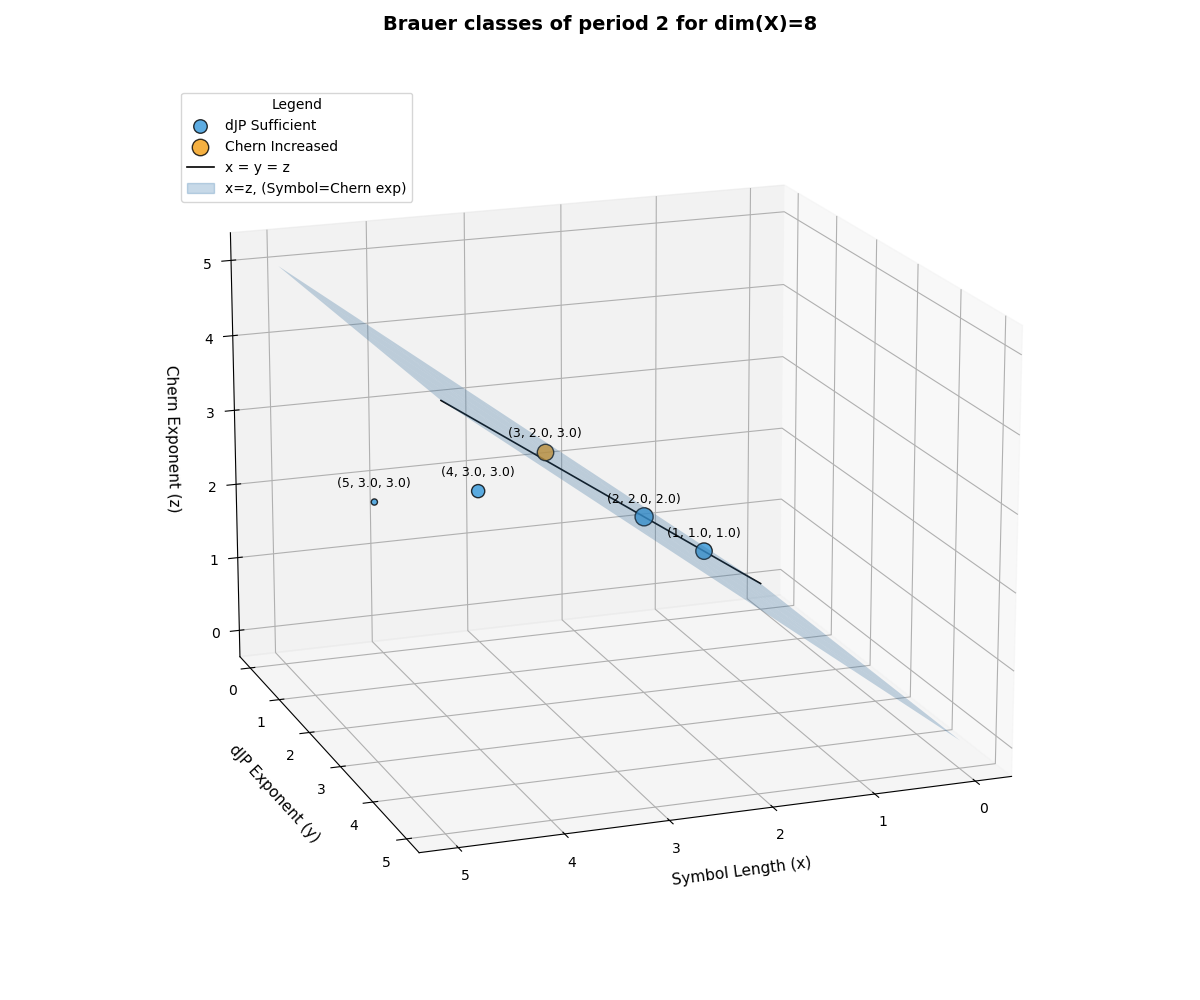}
    \end{adjustbox}
    \caption{A collection of 56 Brauer classes of period 2 on a very general abelian 8-fold sampled from 2-forms of Hamming weight at most 6, rotated.}
    \label{fig: bfield_data_8_2_r}
\end{figure}

\end{document}